\documentclass{amsart}
\usepackage{paralist}
\usepackage{mathtools}
\usepackage{tikz}
\usepackage[nobysame]{amsrefs}
\usepackage{amsmath,amssymb,amsthm,color,mathrsfs,array}
\usetikzlibrary{arrows.meta, positioning}
\usepackage{enumitem}
\usepackage{multirow}
\usepackage{verbatim}
\usepackage{tabularx}
 \usepackage[verbose]{placeins}
\usepackage{pgfplots}
\pgfplotsset{compat=1.18}

\newcommand{\re}{\mathbb{R}}
\newcommand{\mR}{\mathbb{R}}

\newcommand{\mC}{\mathbb{C}}

\newcommand{\N}{\mathbb{N}}

\newcommand{\lmd}{\lambda}

\def\af{\alpha}

\def\rank{\mbox{rank}}

\newcommand{\sig}{\sigma}

\newcommand{\reff}[1]{(\ref{#1})}

\newcommand{\mc}[1]{\mathcal{#1}}

\newcommand{\qmod}[1]{\mbox{QM}[#1]}
\newcommand{\ideal}[1]{\mbox{Ideal}[#1]}
\newcommand{\cred}[1]{\textcolor{red}{#1}}
\newcommand{\st}{\mathit{s.t.}}

\newcommand{\bdes}{\begin{description}}
	\newcommand{\edes}{\end{description}}
\newcommand{\bal}{\begin{align}}
	\newcommand{\eal}{\end{align}}
\newcommand{\bnum}{\begin{enumerate}}
	\newcommand{\enum}{\end{enumerate}}
\newcommand{\bit}{\begin{itemize}}
	\newcommand{\eit}{\end{itemize}}
\newcommand{\bea}{\begin{eqnarray}}
	\newcommand{\eea}{\end{eqnarray}}
\newcommand{\be}{\begin{equation}}
	\newcommand{\ee}{\end{equation}}
\newcommand{\baray}{\begin{array}}
	\newcommand{\earay}{\end{array}}
\newcommand{\bsry}{\begin{subarray}}
	\newcommand{\esry}{\end{subarray}}
\newcommand{\bca}{\begin{cases}}
	\newcommand{\eca}{\end{cases}}
\newcommand{\bcen}{\begin{center}}
	\newcommand{\ecen}{\end{center}}
\newcommand{\bbm}{\begin{bmatrix}}
	\newcommand{\ebm}{\end{bmatrix}}
\newcommand{\bmx}{\begin{matrix}}
	\newcommand{\emx}{\end{matrix}}
\newcommand{\bpm}{\begin{pmatrix}}
	\newcommand{\epm}{\end{pmatrix}}
\newcommand{\btab}{\begin{tabular}}
	\newcommand{\etab}{\end{tabular}}
\theoremstyle{plain}

\newtheorem{theorem}{Theorem}[section]

\newtheorem{prop}[theorem]{Proposition}

\newtheorem*{claim*}{Claim}

\newtheorem{thm}[theorem]{Theorem}

\theoremstyle{definition}

\newtheorem{exa}[theorem]{Example}
\newtheorem{alg}[theorem]{Algorithm}

\newtheorem{remark}[theorem]{Remark}

\def\simplex{{\triangle^{m-1}}}

\numberwithin{equation}{section}
\numberwithin{table}{section}

\def\r{{\mathbb{R}}}

\def\mom_k{{\mathscr{P}_{d}(K)}}

\def\n{{\mathbb{N}}}
\def\nd{{\mathbb{N}_d^n}}
\def\n{{\mathbb{N}}}

\begin{document}

\title[Optimization over the weakly Pareto Set]
{Optimization over the weakly Pareto set and Multi-Task Learning}
\date{\today}

\author[Lei Huang]{Lei Huang}
\address{Lei Huang, Jiawang Nie, and Jiajia Wang,
Department of Mathematics, University of California San Diego,
9500 Gilman Drive, La Jolla, CA, USA, 92093.}
\email{leh010@ucsd.edu,njw@math.ucsd.edu,jiw133@ucsd.edu}

\author[Jiawang Nie]{Jiawang Nie}
%\address{Jiawang Nie,  Department of Mathematics,
%	University of California San Diego,
%	9500 Gilman Drive, La Jolla, CA, USA, 92093.}
%\email{njw@math.ucsd.edu}

\author[Jiajia Wang]{Jiajia Wang}


\begin{abstract}
We study the optimization problem over the weakly Pareto set of
a convex multiobjective optimization problem given by polynomial functions.
Using Lagrange multiplier expressions and the weight vector,
we give three types of representations for the weakly Pareto set.
Using these representations, we reformulate the optimization problem over the weakly Pareto set as a polynomial optimization problem. We then apply the Moment--SOS hierarchy to solve it and analyze its convergence properties under certain conditions.
Numerical experiments are provided to demonstrate the effectiveness of our methods.
Applications in multi-task learning are also presented.
\end{abstract}

\maketitle

\section{Introduction}
\label{introduction}

The multiobjective optimization problem (MOP) concerns how to optimize
several objective functions over a common feasible set.
A typical MOP can be formulated as
\be  \label{mop}
\left\{ \baray{rl}
\min & F(x)  \coloneqq  (f_1(x),\dots,f_m(x))  \\
\st &  c(x)  \coloneqq  (c_1(x),\, \ldots,c_l(x) ) \ge 0, \\
\earay \right.
\ee
where all $f_i(x),c_j(x)$ are functions in the decision variable
$x:=(x_1,\dots,x_n)$. Let $K$ denote the feasible set of \reff{mop}.
MOPs have broad applications in economics \cite{economics2},
finance \cite{finance}, machine learning \cites{machinelearning2},
and scenarios involving multiple tasks \cites{multitasksurvey,briefreview}.

Since the objectives $f_i$ may conflict each other, a decision vector $x$
that minimizes all $f_i$ simultaneously generally does not exist.
The concept of (weakly) Pareto optimal solutions
is commonly used to characterize optimal trade-off decision points
\cites{multi,survey2}.
A point $x^{\ast}\in K$ is called a {\it Pareto point} (PP) if there does not exist $x\in K$ such that $f_i(x)\le f_i(x^{\ast})$ for all $i=1,\dots,m$ and $f_j(x)<f_j(x^{\ast})$ for at least one $j$.
That is, at a Pareto point, it is impossible to improve any individual objective without worsening at least one of the others. A point $x^{\ast}\in K$ is called  a {\it weakly Pareto point}
%(WPP)}
(WPP) if there does not exist $x\in K$ such that $f_i(x)<f_i(x^{\ast})$ for all $i=1,\dots,m$.
That is, at a weakly Pareto point,
it is impossible to improve all objectives simultaneously. The set of all Pareto points (resp., weakly Pareto points) forms the Pareto set (resp., the weakly Pareto set).
Clearly, every Pareto point is a weakly Pareto point,
while the converse is not necessarily true. In computational practice, MOPs are often solved
by scalarization techniques, which convert the MOP into a single-objective optimization problem.
Typical scalarization techniques include linear scalarization \cite{lspcite},
the $\epsilon$-constraint method \cite{epsilon1}, Chebyshev scalarization \cite{cheby},
and boundary intersection methods \cite{boundary1}.
More scalarization methods can be found in the surveys \cites{multi2,multi4}.
For MOPs given by polynomials, there exist Moment-SOS relaxation methods;
see \cites{soscite2,soscite3,multisos1,guofeng,multi}.
We refer to \cites{existencets,existencets2} for results on the existence of weakly Pareto points.
%We refer to \cite{survey1,multi,bao1,bao2,existence1}
%\cred{for the conditions under which various types of solutions exist.}

%for the theory of existence of various types of solutions.
%\textbf{(What is existence theory?)}

There are infinitely many (weakly) Pareto points in general.
%generally.
In some application scenarios,
decision-makers often need to select the {\it best}
solution among the set of all (weakly) Pareto points,
based on an additional preference function $f_0$.
This leads to the optimization problem over the weakly Pareto set (OWP):
\be  \label{oop}
\left\{ \baray{rl}
\min & f_0(x)  \\
\st &  x\in \mathcal{WP}, \\
\earay \right.
\ee
where $\mathcal{WP}$ denotes the weakly Pareto set of \reff{mop}, and
$f_0(x)$ is the preference function for weakly Pareto points.
Denote the optimal value of \reff{oop} by $f_{\min}$.
The problem \reff{oop} has important applications,
such as mean-variance portfolio optimization \cite{portfolio},
production planning \cite{production}, and multi-task learning \cite{review}. 
We remark that the Pareto set is generally not closed, 
whereas the weakly Pareto set is always closed 
when $f_i(x), c_j(x)$ are continuous functions \cite{closed}.
 For instance, consider the MOP \reff{mop} with two objective functions:
\[
f_1(x)=x^2-1, ~ f_2(x)=-x^2+2x,
\]
and the feasible set
$
K = \{x\in\r:\,0\le x\le 3\}.
$
The feasible point $x^{\ast}=2$ lies in the closure of the Pareto set 
but it is not a Pareto point since  $f_1(0)<f_1(2)$ and $f_2(0)=f_2(2).$

Since the set $\mathcal{WP}$ is typically hard to characterize \cite{survey1},
solving \reff{oop} is a challenging task. When the objective functions are strictly convex,
an unconstrained optimization problem of the form  \reff{oop} is investigated in \cite{ma}.
In \cite{jiang}, a gradient-based algorithm is given to
approximate the optimal value when the MOP is convex.
In \cite{julien}, necessary optimality conditions are studied
when the MOP is given by quadratic functions. In \cite{frombilevel},
the OWP is studied in the perspective of bilevel optimization.
We refer to \cites{ankur,mejia} for surveys of existing work on this topic.

%We make the following assumptions:
%\begin{enumerate}[label = (\roman*)]
    %\item The objective functions $f_1,\dots,f_m$ are convex polynomials.
    %\item The feasible set $K$ is a convex set.
%\end{enumerate}

\subsection*{Contributions}
This paper studies optimization over the weakly Pareto set in the form  of \reff{oop},
where the functions are given by polynomials.
The MOP \reff{mop} is said to be convex if each objective function $f_i(x)$ is convex
and each constraining function $c_i(x)$ is concave
(hence the feasible set $K$ must also be convex).
When \reff{mop} is convex,
every WPP $x^*\in \mR^n$ is a minimizer of the linear scalarization problem
\be  \nonumber
\left\{ \baray{rl}
\min & w_1f_1(x)+\dots + w_mf_m(x)  \\
\st &  	c_1(x)\ge 0\, \dots,c_l(x)\ge 0, \\
\earay \right.
\ee
for some nonnegative weight vector $w \coloneqq (w_1,\dots,w_m)\geq 0$ satisfying $\sum\limits_{j=1}^{m} w_j = 1$ (see Section \ref{section3}).
Under some suitable constraint qualifications, there exists a Lagrange multiplier vector
$\lambda \coloneqq (\lambda_1,\dots,\lambda_l)$ such that
\begin{equation}\nonumber
\left\{
\begin{array}{l}
\sum\limits_{j=1}^m w_j \nabla f_j(x) = \sum\limits_{i=1}^l \lambda_j\nabla c_i(x),\\
%% \lambda_i \cdot c_i(x)=0, \, i=1, \ldots, l, \\
0 \le c_i(x)  \perp   \lambda_i\ge 0, \, i=1, \ldots, l.
\end{array}
\right.
\end{equation}
In the above, $c_i(x) \perp \lambda_i$ means the product $c_i(x)\cdot \lambda_i=0$.
Then, the set $\mathcal{WP}$ can be equivalently expressed as
\[
\mathcal{WP}=\left\{x\in\r^n\,\begin{array}{|l}
\exists \, (\lmd_1\dots,\lmd_l) \in \mR^{l}, \,
\exists \,  (w_1\dots,w_m) \in \mR^{m},  \\
\sum\limits_{j=1}^m w_j \nabla f_j(x) = \sum\limits_{i=1}^l \lambda_j\nabla c_i(x),\\
0 \le c_i(x)  \perp   \lambda_i\ge 0, \, i=1, \ldots, l, \\
w_1 \ge 0, \ldots, w_m \ge 0, ~   w_1 + \cdots + w_m = 1
\end{array}\right\} .
\]
Note that $\mathcal{WP}$ is the projection of a higher-dimensional set in $\mR^{n+m+l}$.
Using the above representation,  the OWP \reff{oop} can be recast as an optimization problem  in the decision variable $x$, the weight variable $w$ and the Lagrange multiplier variable $\lambda$.
The Moment-SOS hierarchy of semidefinite relaxations introduced by Lasserre \cite{158}
can be applied to solve it. However, the size of this hierarchy heavily depends
on the number of extra variables $w$ and $\lambda$,
which significantly increase the computational expense.
In this paper, we explore more computationally efficient ways to express the set $\mathcal{WP}$.
Specifically, we study three types of representations for the set $\mathcal{WP}$.

An interesting class of MOPs is given by the objectives such that
\begin{equation}\label{spe:MOP}
f_i(x)=h(x)+ d_i^Tx, \quad i=1,\ldots,m,
\end{equation}
where $h(x)$ is a common convex polynomial function, and the vectors
$d_j\in \mR^n$ are typically different. That is, the objectives $f_i$
only differ in linear terms. This kind of MOPs have
broad applications in multiobjective  linear programming \cite{bvpe},
portfolio optimization \cite{portfolio2},
minimizing energy consumption and costs in supply chain operations \cite{chain}.
For this class of MOPs, we can obtain highly efficient representations for the set $\mc{WP}$.

Our major contributions are:

\bit

\item We give efficient characterizations for the weakly Pareto set $\mc{WP}$.
Under different nonsingularity assumptions, we show how to express the weakly Pareto set
$\mc{WP}$ in terms of  $(x,w)$, or $(x,\lambda)$, or solely in $x$.
This leads to three types of representations for $\mc{WP}$.
%This is done in Section \ref{section3}.

\item Using the representations for $\mc{WP}$, we reformulate the OWP \reff{oop}
as a polynomial optimization problem
and apply the Moment-SOS hierarchy to solve it. Under some conditions,
we study how to extract optimizations for the OWP from this hierarchy.
Numerical experiments are given to demonstrate
the efficiency of our methods.

\item We show the applications of OWP in multi-task learning problems in machine learning.
Global optimizers for these problems can be obtained by
the reformulated polynomial optimization problem.

\eit

%
%\begin{comment}
%Under some general optimality conditions, we show that a point is feasible for the OWP if and only if it is a KKT point of certain linear scalarization problems (LSPs).
%\cred{not proved by us?}
%\end{comment}
%

The paper is organized as follows. Section \ref{section2} reviews some basics in polynomial
optimization.  In Section \ref{section3},
we give three types of representations for the weakly Pareto set
$\mc{WP}$ when the MOP \reff{mop} is convex.
Section \ref{section4} discusses how to apply the Moment-SOS hierarchy to solve the OWP \reff{oop}.
Some numerical examples for the OWP are given in Section \ref{section5}.
Section \ref{section6} presents the applications in multi-task learning.
Some conclusions and discussions are made in Section~\ref{sec:dis}.

\section{Preliminaries}\label{section2}
\label{preliminaries}

\subsection*{Notation}

The symbol $\n$ (resp., $\r$, $\mC$) denotes the set of nonnegative integers
(resp., real numbers, complex).  The notation $e_i$ denotes the
$i$th unit vector, which has 1 in the
$i$th entry and 0 in all other entries.
The $e$ denotes the vector of all ones.
For two scalars $a,b$, the notation $a\perp b$ means that $a\cdot b=0$. The notation $\|x\|$ denotes the  2-norm of the vector $x$ and the notation $\|A\|_F$ denotes the Frobenious norm of the matrix $A$.
For $t \in \mathbb{R}$, $\lceil t \rceil$
denotes the smallest integer greater than or equal to $t$.
%For a matrix $A$, $A^{\mathrm{T}}$ denotes its transpose.
A symmetric matrix $X \succeq 0$  if $X$ is positive semidefinite.
For a smooth function $f(x)$,
denote by $\nabla f(x)$ its gradient with respect to $x$.
In particular, $\nabla_{x_k} f(x)$ denotes the partial derivative with respect to the variable $x_k$.

Let $\r[x]$ denote the polynomial ring in  $x$, and let $\r[x]_k$
denote the subring of $\r[x]$ consisting of polynomials with degree at most  $k$.
Denote by $\deg(p)$  the total degree of the polynomial $p$.
For a positive integer $l$, the notation $[l]$ represents the set $\{1,\dots,l\}$,
and the notation $I_l$ denotes the $l\times l$ identity matrix.
For $x = (x_1,\dots,x_n)$ and $\af = (\af_1, \ldots, \af_n)$, denote
\[
x^{\alpha} \coloneqq x_1^{\alpha_{1}}\cdots x_n^{\alpha_{n}}, \quad
|\alpha| \coloneqq  \alpha_{1}+\cdots+\alpha_{n}.
\]
The power set of degree $d$ is
$$
\mathbb{N}_d^n:=\left\{\alpha \in \mathbb{N}^n: \,\,|\alpha| \leq d\right\}.
$$

%For each label $i$, the $e_i$ denotes the vector of all zeros excepts its $i$ th entry being $1$, and $e$ denotes the vector of all ones.

\subsection{Some basics in polynomial optimization}
In this subsection, we review some basics of polynomial optimization and moment theory.
For more details, we refer the reader to \cites{sos,nie2023moment,matrix,las2,marshall}.
A polynomial $\sig \in \mathbb{R}[x]$ is said to be a sum of squares $(\mathrm{SOS})$
if $\sig =p_{1}^{2}+\cdots+p_{k}^{2}$ for some $p_{1}, \ldots, p_{k} \in \mathbb{R}[x].$
The cone of all SOS polynomials is denoted as $\Sigma[x].$
For a given degree $k\in\n$,
we denote the $k$th truncation of $\Sigma[x]$ by
%we denote its $k$th truncation by
\[
\Sigma[x]_{k} \coloneqq  \Sigma[x] \cap \mathbb{R}[x]_{k}.
\]
A subset $I$ of $\r[x]$ is an ideal if $I\cdot \r[x]\subseteq I$  and $I+I\subseteq I$. For a polynomial tuple $h =  (h_{1}, \ldots, h_{m_1} )$,
the ideal generated by $h$ is defined as
\[
\ideal{h} \coloneqq  h_{1} \cdot \mathbb{R}[x]+\cdots+h_{m_1} \cdot \mathbb{R}[x].
\]
The $k$th degree truncation of $\ideal{h}$ is
\begin{equation}
\ideal{h}_{k} \coloneqq h_{1} \cdot \mathbb{R}[x]_{k-\deg\left(h_{1}\right)}+
   \cdots+h_{m_1} \cdot \mathbb{R}[x]_{k-\deg\left(h_{m_1}\right)}.
\end{equation}
For a polynomial tuple $g \coloneqq \left(g_{1}, \ldots, g_{m_2}\right)$,
the quadratic module generated by $g$ is
\begin{equation*}
\qmod{g} \coloneqq \Sigma[x]+g_{1} \cdot \Sigma[x]+\cdots+g_{m_2} \cdot \Sigma[x].
\end{equation*}
Similarly, the $k$th degree truncation of $\qmod{g}$ is
\begin{equation}
\qmod{g}_k \coloneqq \Sigma[x]_{k}+g_{1} \cdot \Sigma[x]_{k-\deg(g_1)}+\cdots+g_{m_2} \cdot \Sigma[x]_{k-\deg(g_{m_2})}.
\end{equation}
The set $\ideal{h}+\qmod{g}$ is said to be  Archimedean if there exists $R>0$ such that $R-\|x\|^2\in \ideal{h}+\qmod{g}.$

For a given degree $d$, let $\re^{ \N^n_{d} }$ denote the set of all real vectors $y$ labeled by $\af \in \N^n_{d}$. Each
$y \in \re^{ \N^n_{d} }$ can be represented as
$
y \, = \, (y_\af)_{ \af \in \N^n_{d} },
$
and it is called a
{\it truncated multi-sequence} (tms) of degree $d$.
%The $d$th degree moment cone for a set $K$ is defined as
%\[
%\mathscr{R}_d(K):= \left\{
%y\in \r^{\n^n_d}\,
%\begin{array}{|c}
%     y_{\alpha}=\int x^{\alpha} d\mu \text{ for each $\alpha \in \n^{n_d}$},  \\
%      \operatorname{supp}(\mu)\subseteq K
%\end{array}
%\right\}.
%\]
A tms $y \in \mathbb{R}^{\nd}$  defines a bilinear operation on $ \mathbb{R}[x]_{d}$ as follows:
\begin{equation}
\langle \sum_{\alpha\in\nd} p_{\alpha}x^{\alpha}, y \rangle =
\sum_{\alpha\in\nd} p_{\alpha}y_{\alpha}.
\end{equation}
For a polynomial $q \in \mathbb{R}[x]_{2k}$ and  $y \in \mR^{\mathbb{N}^n_{2k}}$,
the $k$th order localizing matrix of $q$ generated by $y$ is the symmetric matrix
${L}_{q}^{(k)}[y]$ satisfying
\begin{equation} \label{exp:sdr}
	\left\langle q \cdot p^2, y\right\rangle=\operatorname{vec}
\left(p\right)^{T}\left({L}_{q}^{(k)}[y]\right)
\operatorname{vec}\left(p\right)\quad \forall \,
p \in \mathbb{R}[x]_{k-\lceil \deg(q)/2\rceil} .
\end{equation}
Here, $\operatorname{vec}\left(p\right)$ denotes the coefficient vector of $p$
in graded lexicographical order. In particular, when $q=1$,
the localizing matrix ${L}_{q}^{(k)}[y]$ reduces to the
$k$th order moment matrix, for which we denote by $M_k[y]$. For instance, when $n=2$, the  moment matrix $M_2[y]$ is
\[
M_2[y]=\begin{bmatrix}
    y_{00}&y_{10}&y_{01}&y_{20}&y_{11}&y_{02}\\
    y_{10}&y_{20}&y_{11}&y_{30}&y_{21}&y_{12}\\
    y_{01}&y_{11}&y_{02}&y_{21}&y_{12}&y_{03}\\
    y_{20}&y_{30}&y_{21}&y_{40}&y_{31}&y_{22}\\
    y_{11}&y_{21}&y_{12}&y_{31}&y_{22}&y_{13}\\
    y_{02}&y_{12}&y_{03}&y_{22}&y_{13}&y_{04}
\end{bmatrix},
\]
and the  localizing matrix $L^{(2)}_{q}[y]$ for  $q=x_1^2-x_2$ is
\[
L^{(2)}_{x_1^2-x_2}[y]=\begin{bmatrix}
    y_{20}-y_{01}&y_{30}-y_{11}&y_{21}-y_{02}\\
    y_{30}-y_{11}&y_{40}-y_{21}&y_{31}-y_{22}\\
    y_{21}-y_{02}&y_{31}-y_{12}&y_{22}-y_{03}
\end{bmatrix}.
\]

\subsection{The Moment-SOS hierarchy for polynomial optimization}
\label{subsection 2.3}

In this subsection, we introduce  the Moment-SOS hierarchy of semidefinite relaxations for solving polynomial optimization; see \cites{158,sos,generalized} for more details. Consider the polynomial optimization problem:
\begin{equation}
    \label{standard polynomial problem}
    \left\{
    \begin{array}{cl}
        \min & f(x) \\
        \st &  c_i(x)=0\,(i\in \mathcal{I}),
        \\ &c_i(x)\ge 0\,(i\in\mathcal{J}),
    \end{array}
    \right.
\end{equation}
where  $f, c_i, c_j\in \r[x]$, and $\mathcal{I}$,  $\mathcal{J}$ are the index sets of equality and inequality constraints, respectively.
%Denote the polynomial tuples
%$$c_{eq}:=\{c_i:\,\,i\in\mathcal{I}\},\quad c_{in}:=\{c_j:\,\,j\in\mathcal{J}\}.
%$$
%We refer to the book \cite{nie2023moment} for more details.
For a relaxation order $k$, the  $k$th order SOS relaxation for solving
\reff{standard polynomial problem} is
\begin{equation}
\label{sosgeneral}
\left\{
\begin{array}{cl}
\max   & \gamma \\
\st & f-\gamma \in \ideal{(c_i)_{i\in\mathcal{I}}}_{2k}+ \qmod{(c_i)_{i\in\mathcal{J}}}_{2k}.
\end{array}
\right.
\end{equation}
The dual optimization of \reff{sosgeneral} is the $k$th order moment relaxation
\begin{equation}
\label{momgeneral}
\left\{
\begin{array}{cl}
\min & \langle f,y\rangle \\
\st &\langle 1,y\rangle=1,\\
& L_{c_{j}}^{(k)}[y]=0 \, (j\in \mathcal{I}),\\
& L_{c_j}^{(k)}[y]\succeq 0\, (i\in\mathcal{J}),\\
& M_k[y]\succeq 0,\,y\in \mR^{n}_{2k}.    \\
\end{array}
\right.
\end{equation}
For $k=1,2,\dots$, the sequence of primal-dual  relaxation pairs \reff{sosgeneral}-\reff{momgeneral}
is referred to as the Moment-SOS hierarchy. Under the Archimedean property,
it was shown in \cite{158} that this hierarchy has asymptotic convergence.
For results on finite convergence, we refer to
\cites{huang2023,hny,hny2,niecondition}.

\section{Representations for weakly Pareto points}
\label{section3}

In this section, we characterize the weakly Pareto set $\mc{WP}$ for convex MOPs.
It can be expressed through the Lagrange multiplier vector $\lambda$
and the weight vector $w$ for the linear scalarization optimization.

%\subsection{Characterize the WPPs} \label{Sec:propWPP}

MOPs are often solved using linear scalarization, which optimizes
a nonnegative linear combination of individual objectives \cites{multi,multi2, multi4}.
Denote the $(m-1)$-dimensional simplex of vectors $w:=(w_1,\dots,w_m)$:
\[
\Delta^{m-1}  \coloneqq \{w \in\r^m:\,
w \ge 0, \,  w_1+\dots+w_m=1   \}.
\]
For a weight vector $w\in \Delta^{m-1}$, the linear scalarization problem (LSP) for \reff{mop} is
\be  \label{lsp}
\left\{ \baray{rl}
\min & f_w(x)  \coloneqq   w_1f_1(x)+\dots + w_mf_m(x)  \\
\st &  	c_1(x)\ge 0\, \dots,c_l(x)\ge 0. \\
\earay \right.
\ee
Clearly, every minimizer $x^*$ of \reff{lsp} is a weakly Pareto point of  \reff{mop}
for each $w\in \Delta^{m-1}$, and $x^*$ is a Pareto point if $w>0$.
By choosing different weight vectors $w$, we may obtain different weakly Pareto points.
When \reff{mop} is nonconvex, not every WPP is the minimizer of a scalarization problem.
However, if the MOP \reff{mop} is convex, this is true, which is well-known in multiobjective optimization \cite{review,vectoropt}. Below, we summarize these results and provide direct proofs for the convenience of the reader.
%The following theorem is well-known in the field of multiobjective optimization.
%We give a direct proof for convenience of readers.

\begin{thm} \label{theorem1}
Suppose that the MOP \reff{mop} is convex.
Then, we have

\bit
\item[(i)] A point $x^* \in K$ is a weakly Pareto point of \reff{mop}
if and only if $x^*$ is a minimizer of the LSP \reff{lsp}
for some weight vector $w\in \Delta^{m-1}$.

\item[(ii)] If every objective $f_i$ is strictly convex, every weakly Pareto point of \reff{mop} is a Pareto point.

\eit
\end{thm}
\begin{proof}
(i) The ``if" part is obvious, we omit the proof for cleanness. For the ``only if" direction, let
\[
\mathcal{U}   \coloneqq  \{u=(u_1,\dots,u_m)\in \mR^m: u_i\ge f_i(x) \text{ for some }x\in K\}.
\]
One can see that $\mathcal{U}$ is a convex set in $\mR^m$ since $f_i$ is convex.
Let $x^* \in K$ be a WPP of \reff{mop}.
Then, the vector $F(x^*)=(f_1(x^*),\dots,f_m(x^*))$ must lie on the boundary of $\mathcal{U}$.
This is because if $F(x^*)$ were an interior point of  $\mathcal{U}$,
there would exist
 $x\in K$ and $v = (v_1, \ldots, v_m) \in \mR^m$ satisfying $v_i > 0~(i\in [m])$
such that $f_i(x)+v_i < f_i(x^*)$ for all $i\in [m]$. This contradicts the assumption that  $x^*$ is a WPP.
By the hyperplane separation theorem, there exists a nonzero vector $w^* \in \mR^m$ such that
\[
\langle w^*,F(x^*)\rangle \le \langle w^*, u \rangle\text{~~for all~~} u\in\mathcal{U}.
\]
By construction of $\mc{U}$, the vector $w^*$ must be nonnegative.
Up to a positive scaling, we can assume $w^*_1+\cdots+w^*_m=1$.
The above relation implies that
\[
\langle w^*,F(x^*)\rangle \le \langle w^*, F(x) \rangle \quad \text{for all} \quad  x\in K .
\]
This means that $x^*$ is a minimizer of \reff{lsp} for the weight vector $w^*$.

(ii) Suppose, on the contrary, that $x^*\in \mR^n$ is a weakly Pareto point but not a Pareto point.
Then, there exists a point $v\in K$ such that $f_i(v)\le f_i(x^*)$ for all $i\in [m]$
and $f_j(v)<f_j(x^*)$ for some $j$. Clearly, $v \ne x^*$.
Since each $f_i$ is strictly convex and $K$ is convex,  it follows that
$\lambda x^*+(1-\lambda )v\in K$ for all $\lambda \in (0,1)$ and
\[
f_i(\lambda x^*+(1-\lambda )v)<\lambda f_i(x^*)+(1-\lambda) f_i(v) \le f_i(x^*),
\quad i=1,\ldots, m.
\]
This contradicts the assumption that $x^*$ is a weakly Pareto point.
\end{proof}

\begin{thm}
Suppose the functions $f_1,\dots,f_m,c_1,\dots,c_l$ as in the MOP \reff{mop} are continuous, then the weakly Pareto set $\mathcal{WP}$ is closed.
\end{thm}
\begin{proof}
Suppose  there exists a sequence of weakly Pareto points $\{x_k\}_{k=1}^{\infty}\in \mathcal{WP}$ such that $x_k \rightarrow x^{\ast}$. If $x^{\ast}\notin \mathcal{WP}$, there exists $x'\in K$ such that $f_i(x') < f_i(x^{\ast})$ for all $i\in[m]$. Since $x_k\rightarrow x^{\ast}$ and  $f_i$ is continuous,  we have $\lim\limits_{k\rightarrow \infty}f_i(x_k)= f_i(x^{\ast})$ for $i\in[m]$.
%Let $\delta_i = f_i(x^{\ast})-f_i(x')$ for $i\in [n]$, we can choose $\epsilon_i$ for $i\in[n]$ such that $0<\epsilon_i<\frac{\delta_i}{2}.$ For each $i$, there exists $k_i\in\mathbb{N}$ such that if $k\ge k_i$, then $|f_i(x^{\ast})-f_i(x_k)|<\epsilon_i<\frac{\delta_i}{2}$. Choose $k_{max}:= \max \{k_1,\dots,k_n\}$,
It holds that $f_i(x_k)>f_i(x')$ for all $i$ for $k$ sufficiently large, which is a contradiction.
\end{proof}

\hspace{1cm}

In the following, we give three types of representations for
the weakly Pareto set $\mathcal{WP}$, based on Theorem~\ref{theorem1}.

\subsection{Representation of the set $\mathcal{WP}$ in terms of $x,w$}
\label{rep1}

In this subsection, we show how to express the Lagrange multiplier vector
$\lambda$ in terms of $x$ and $w$.
When the MOP \reff{mop} is convex, it follows from Theorem \ref{theorem1} that
a point $x \in K$ is a WPP if and only if $x$ is a minimizer of the LSP \reff{lsp},
for some weight vector $w\in \Delta^{m-1}$. Under certain constraint qualifications
(e.g., linear independence constraint qualification, Slater's condition),
there exists a Lagrange multiplier vector
$\lambda \coloneqq (\lambda_1,\dots,\lambda_l) \in \mR^l$ such that
\begin{equation}\label{sec3:kkt1}
\left\{
\begin{array}{l}
\sum\limits_{j=1}^m w_j\nabla f_j(x) = \sum\limits_{i=1}^l \lambda_i\nabla c_i(x),\\
 0 \le c_i(x) \perp \lambda_i\ge 0, \, i = 1,\ldots, l.
\end{array}
\right.
\end{equation}
When the MOP \reff{mop} is convex, every point $x\in \mR^{n}$  satisfying
\reff{sec3:kkt1} is aslo a minimizer of \reff{lsp}.
Thus, the weakly Pareto set $\mathcal{WP}$ can be described as
\[
\mathcal{WP}=\left\{x\in\r^n\,\begin{array}{|l}
\exists\, \lambda:=(\lambda_1,\dots,\lambda_l)\in\r^l, \\
\exists\, w \coloneqq (w_1,\dots,w_m)\in \Delta^{m-1},   \\
\sum\limits_{j=1}^{m} w_j \nabla f_j(x) =\sum\limits_{i=1}^{l}\lambda_i c_i(x),\\
  0 \le c_i(x) \perp \lambda_i\ge 0, \, i = 1,\ldots, l, \\
% w\ge 0,\,\sum\limits_{j=1}^m w_j=1 \\
\end{array}\right\}.
\]
The above representation for $\mc{WP}$ requires to use extra variables $\lambda$ and $w$.
%which increases the computational cost.
Interestingly, the Lagrange multiplier vector $\lambda$ can be eliminated for general cases.
The equation \reff{sec3:kkt1} implies that
\begin{equation}
 \label{sec3:def:C}
\underbrace{\left[\begin{array}{cccc}
\nabla c_{1}(x) & \nabla c_{2}(x) & \cdots & \nabla c_{l}(x) \\
c_{1}(x) & 0 & \cdots & 0 \\
0 & c_{2}(x) & \cdots & 0 \\
\vdots & \vdots & \ddots & \vdots \\
0 & 0 & \cdots & c_{l}(x)
\end{array}\right]}_{C(x)} \underbrace{\left[\begin{array}{c}
\lambda_{1} \\
\lambda_{2} \\
\vdots \\
\lambda_{l}
\end{array}\right]}_{\lambda}=\underbrace{\left[\begin{array}{c}
\sum\limits_{j=1}^{m} w_j \nabla f_j(x) \\
0 \\
\vdots \\
0
\end{array}\right]}_{\hat{f}_{w}(x)}
\end{equation}

The polynomial tuple $c = (c_1,\dots,c_l)$ is said to be {\it nonsingular} if
$\operatorname{rank}C(x)=l$ for all $x\in\mathbb{C}^n$. When $c$ is nonsingular,
there exists a matrix polynomial $C^{\prime}(x)\in \mR[x]^{l\times (n+l)}$ such that
$C^{\prime}(x)C(x) = I_l$ (see \cite{nietight}*{Proposition~5.1}, then
\begin{equation}
\label{ref:L1}
\lambda  =  \lmd(x, w) \coloneqq  C^{\prime}(x)\hat{f}_{w}(x)
= \sum\limits_{j=1}^{m} w_j C^{\prime}_1(x)\nabla f_j(x),
\end{equation}
where $C^{\prime}_1(x)$ is the submatrix consisting of its first $n$ columns.
The $i$th entry of $\lmd(x, w)$ is denoted as $\lmd_i(x, w)$, i.e.,
\be \label{expr:lmd=(x,w)}
\lmd_i(x, w)  =  \sum\limits_{j=1}^{m} w_j e_i^T C^{\prime}_1(x)\nabla f_j(x),
\ee
\[
\lambda(x,w) \, = \, (\lambda_1(x,w),\dots,\lambda_l(x,w)).
\]
The weakly Pareto set $\mathcal{WP}$ can be represented in terms of $(x,w)$ as
\be  \label{rep:wpp}
\mathcal{WP}=\left\{x\in\r^n\,
\begin{array}{|l}
\exists\, w:= ( w_1,\dots,w_m)\in \r^m  ,   \\
\sum\limits_{j=1}^{m} w_j \nabla f_j(x) = \sum\limits_{i=1}^{l}\lambda_i(x,w)\nabla c_i(x),  \\
   0 \le c_i(x) \perp \lambda_i (x, w) \ge 0, \, i = 1,\ldots, l,\\
   w_1 \ge 0, \ldots,  w_m \ge 0, \,  w_1 + \cdots + w_m  =  1  .   \\
\end{array}\right\}.
\ee

\begin{exa}
\label{cube ref}
(i) Suppose the feasible set $K$ is  the $n$-dimensional hypercube
$$
K = \{x\in\r^n: a_1^2-x_1^2\ge 0,\dots, a_n^2-x_n^2\ge 0\},
$$
for a real vector $a = (a_1, \ldots, a_n)>0$.
One can check that $C^{\prime}(x)C(x)=I_n$ for
\[
C^{\prime}(x) =
\begin{bmatrix}
-\frac{x_1}{2a_1^2}&0&\dots&0&\frac{1}{a_1^2}&0&\dots&0\\
0& -\frac{x_2}{2a_2^2}&\dots&0&0&\frac{1}{a_2^2}&\dots&0\\
\vdots&\vdots&\vdots&\vdots&\vdots&\vdots&\vdots&\vdots\\
0&0&\dots&-\frac{x_n}{2a_n^2}&0&0&\dots&\frac{1}{a_n^2}
\end{bmatrix} .
\]
The polynomial expressions for $\lambda_i$ are
\[
\lambda_i(w,x) = -\frac{x_i}{2a_i^2}(w_1\frac{\partial f_1}{\partial x_i}+
\cdots+w_m\frac{\partial f_m}{\partial x_i}), \quad i=1,\ldots, n.
\]
(ii) Suppose the feasible set $K$  is defined by linear inequalities, i.e.,
\[
K  = \{x\in \r^n:\,a_i^Tx-b_i\ge 0,\, i = 1,\ldots, l \},
\]
where $a_i\in\r^n$ and $b_i\in \r.$
If the vectors $a_i,\dots,a_l$ are linearly independent,
the matrix $C^{\prime}_1(x)$ as in \reff{ref:L1} is given as
\[
C^{\prime}_1(x)=(AA^T)^{-1}A,
\]
where $A = \begin{bmatrix}a_1&\dots &a_l\end{bmatrix}.$  \\
(iii) Suppose the feasible set $K$  is given as
\[
K=\left\{x\in\r^n:\,\alpha_ix_i+q_i(x_{i+1},\dots,x_n)\ge 0,\, i=1,\dots, l\right\},
\]
where $0\ne\alpha_i\in\r$, $q_i\in \mR[x_{i+1},\dots,x_n]$.
Note that the matrix $T(x)$, consisting of the first $l$ rows of
$\,[\nabla c_1(x),\dots,\nabla c_l(x)]$, is an invertible
lower triangular matrix with constant diagonal entries. Hence, we have
\[
\lambda(x,w) = T(x)^{-1} \Big( \sum\limits_{j=1}^m w_j \nabla f_j(x) \Big)_{1:l}.
\]
Here, the subscript $1:l$ denotes the subvector
consisting of the entries indexed from $1$ through $l$.
\end{exa}

\subsection{Representation of the set $\mathcal{WP}$ in terms of $x,\lambda$}
\label{rep2}

In this subsection, we show how to represent the weakly Pareto set $\mathcal{WP}$
in terms of $x$ and $\lambda$. The equation \reff{sec3:kkt1} implies that
\begin{equation}
 \label{sec3:def:D}
\underbrace{\left[\begin{array}{cccc}
\nabla  f_{1}(x) & \nabla f_{2}(x) & \cdots & \nabla f_{m}(x) \\
1 & 1 & \cdots & 1 \\
\end{array}\right]}_{Q(x)} \underbrace{\left[\begin{array}{c}
w_{1} \\
w_{2} \\
\vdots \\
w_{m}
\end{array}\right]}_{w}=\underbrace{\left[\begin{array}{c}
\sum\limits_{i=1}^{l} \lambda_i \nabla c_i(x) \\
1 \\
\end{array}\right]}_{\hat{c}_{\lambda}(x)} .
\end{equation}
The matrix $Q(x)$ defined  above is a $(n+1)$-by-$m$ polynomial matrix.
If $Q(x)$ is nonsingular (i.e., $\rank \, Q(x) = m$ for all $x \in \mC^n$),
then there exists a  polynomial matrix $Q^{\prime}(x)\in \mR[x]^{m\times (n+1)}$ such that
$Q^{\prime}(x)Q(x) = I_{m}$, so
\be
w  = w(x, \lmd)  \coloneqq  Q^{\prime}(x)\hat{c}_{\lambda}(x)
%=Q^{\prime}_1(x)\sum\limits_{i=1}^{l} \lambda_i \nabla c_i(x)+Q^{\prime}_{2}
= \sum\limits_{i=1}^{l} \lambda_i Q^{\prime}_1(x)\nabla c_i(x)+Q^{\prime}_2(x) .
\ee
In the above, $Q^{\prime}_1(x)$ is the submatrix of $Q^{\prime}(x)$
consisting of its first $n$ columns and $Q^{\prime}_2(x)$
is the $(n+1)$-th column of $Q^{\prime}(x)$.
The polynomial vector
\be \label{expr:w(x,lmd)}
w(x,\lambda) = (w_1(x,\lambda), \ldots, w_m(x,\lambda))
\ee
is an expression for the weight vector $w$ in terms of $x$ and $\lambda$,
where $w_j(x,\lambda)$ denotes the $j$th entry of $w(x,\lmd)$.
The weakly Pareto set $\mathcal{WP}$ can be equivalently given as
\begin{equation}\label{rep:wpp2}
\mathcal{WP}=\left\{x\in\r^n\,\begin{array}{|l}
\exists\, \lambda:=(\lambda_1,\dots,\lambda_l)\in\r^l ,    \\
   \sum\limits_{j=1}^{m} w_j(x,\lambda) \nabla f_j(x) =\sum\limits_{i=1}^{l}\lambda_i\nabla c_i(x),\\
    0 \le c_i(x) \perp \lambda_i \ge 0, \, i = 1,\ldots, l, \\
    w_1(x,\lambda) \ge 0, \ldots, w_m(x,\lambda) \ge 0,   \\
    w_1(x,\lambda) + \cdots + w_m(x,\lambda) =  1  \\
\end{array}\right\}.
\end{equation}

An interesting class of MOPs is one in which the objectives are given in the form
\begin{equation}\label{spe:MOP_new}
f_i(x)=h(x)+ d_i^Tx , \quad  i=1,\ldots,m ,
\end{equation}
where $h(x)$ is a common convex polynomial function and $d_1,\dots,d_m\in \mR^n$.
The objectives only differ in linear terms.
The expression \reff{rep:wpp2} can be further simplified for this kind of MOPs.

\begin{prop} \label{thm33}
For the MOP \reff{mop}, suppose the objectives are given in \reff{spe:MOP_new}
and the vectors $d_1,\dots,d_m$ are linearly independent.
Then, for each $x \in \mc{WP}$, the weight vector $w$ can be expressed as
\be  \label{rep:w}
w(x, \lmd) \ = \  (D^TD)^{-1}D^T \Big( \hat{h}_1(x)  + \hat{h}_2(x, \lambda)  \Big),
\ee
where
\be   \label{sec3:def:D2}
\baray{l}
D = \left[\begin{array}{cccc}
 d_{1} & d_{2} & \cdots &  d_{m} \\
1 & 1 & \cdots & 1 \\ \end{array}\right],  \\
\hat{h}_1(x) = \begin{bmatrix}-\nabla h(x)\\1\end{bmatrix}, \,
\hat{h}_2(x, \lambda) =
\begin{bmatrix}\sum\limits_{i=1}^l \lambda_i\nabla c_i(x)\\0\end{bmatrix}.
\earay
\ee
\end{prop}
%%%%%%%%%%%%%%%%%%%%%%%%%%%%%%
\begin{proof}
For this class of MOPs, the equation \reff{sec3:def:D} becomes
\begin{equation}
 \label{sec3:def:D3}
\left[\begin{array}{cccc}
\nabla h(x)+ d_{1} & \nabla h(x)+d_{2} & \cdots & \nabla h(x)+d_{m} \\
1 & 1 & \cdots & 1 \\
\end{array}\right]
%\underbrace{\left[\begin{array}{c}
%w_{1} \\
%w_{2} \\
%\vdots \\
%w_{m}
%\end{array}\right]}_{w}
w =\begin{bmatrix}
\sum\limits_{i=1}^l \lambda_i\nabla c_i(x)\\1
\end{bmatrix} .
\end{equation}
Using row elimination, we can get
\begin{equation}
Dw  =  \hat{h}_1(x)+\hat{h}_2(x, \lmd).
\end{equation}
Since the vectors $d_1,\dots,d_m$ are linearly independent,
the matrix $D$ has full column rank,
so the above implies the expression \reff{rep:w}.
\end{proof}

%Then, the weakly Pareto set can be described as
%\begin{equation}\label{rep:wpp3}
%    \mathcal{WP} =\left\{
%    x\in\r^n\,
%    \begin{array}{|l}
%        \cred{\exists \,\lambda:=(\lambda_1,\dots,\lambda_l)\in\r^l \text{ such that }} \\
%        D w(x,\lmd)  =  \hat{h}_1(x)+\lambda \cdot \hat{h}_2(x), \\
%        1-\|x\|^2 \ge 0, ~~
 %       \lambda\geq 0,~~\lambda\cdot (1-\|x\|^2)=0,\\
 %       w(x,\lmd )\ge 0,\,\sum\limits_{j=1}^m w_j(x,\lmd)=1
 %   \end{array}
%    \right\}.
%\end{equation}

\subsection{Representation of the set $\mathcal{WP}$ in terms of $x$ only}
\label{rep3}

%Note that the reformulation \reff{equ1:oop} of the OWP \reff{oop}
%is a polynomial optimization in the decision vector $x$ and the weight vector $w$,
%and the Lagrange multiplier vector $\lambda$ is eliminated.
In this subsection, we discuss how to express
the weakly Pareto set $\mathcal{WP}$ solely in terms of $x$.
When the MOP \reff{mop} is convex and the matrix $C(x)$ as in \reff{sec3:def:C}  is nonsingular,
the set $\mathcal{WP}$ can be given as in \reff{rep:wpp}.
Let $\ell_i(x)^T$ denote the $i$th row of $C^{\prime}_1(x)$. Then, we have
\be \label{expr:lmdi(x,w)}
\lambda_i(x,w)=\sum\limits_{j=1}^{m} w_j u_{ij}(x) \quad \text{where}\quad
u_{ij}(x) \coloneqq   \ell_i(x)^T \nabla f_j(x).
\ee
For every $x\in \mc{WP}$, it holds that
\[
\sum\limits_{i=1}^{l}\lambda_i(x,w)\nabla c_i(x) = \sum\limits_{i=1}^{l}
    (\sum\limits_{j=1}^{m} w_j   u_{ij}(x))   \nabla c_i(x)\\
= \sum\limits_{j=1}^{m}\sum\limits_{i=1}^{l} w_j u_{ij}(x)  \nabla c_i(x),
\]
\[
\lambda_i(x,w)\cdot c_i(x)=\sum\limits_{j=1}^{m} w_j  u_{ij}(x) c_i(x)=0.
\]
So every WPP of \reff{mop} satisfies the equation
\begin{equation}
\label{def:P}
\underbrace{\begin{bmatrix}
\sum\limits_{i=1}^l u_{i1}(x) \nabla c_i(x) -\nabla f_1(x) & \cdots &
      \sum\limits_{i=1}^l  u_{im}(x) \nabla c_i(x) -\nabla f_m(x)  \\
     u_{11}(x)  c_1(x)  & \cdots  &  u_{1m}(x)  c_1(x)\\
    \vdots&\vdots&\vdots\\
     u_{l1}(x)  c_l(x)&\cdots& u_{lm}(x) c_l(x)  \\
    1 & \dots & 1 \end{bmatrix}}_{P(x)}
%\underbrace{\left[\begin{array}{c}
%w_{1} \\
%w_{2} \\
%\vdots \\
%w_{m}
%\end{array}\right]}_{w}
w = \begin{bmatrix}
        0 \\ 0 \\ \vdots \\ 0 \\  1
    \end{bmatrix} .
\end{equation}

When the matrix polynomial $P(x)$ is nonsingular (i.e.,  $\rank \,P(x) = m$
for all $x\in\mathbb{C}^n$), there exists a polynomial matrix
$P^{\prime}(x)\in \mR[x]^{m\times (n+l+1)}$ such that  $P^{\prime}(x)P(x)=I_m$.
Then, we can get
\be \label{expr:w(x)}
w  =  w(x) \coloneqq  P^{\prime}(x)\cdot e_{n+l+1} .
\ee
We write the above polynomial vector $w(x)$ as
\[
w(x)  =   (w_1(x),\dots,w_m(x)) .
\]
Here we denote by $w_j(x)$ the $j$th component of $P^{\prime}(x)\cdot e_{n+l+1}$.
Further, we have that for each $i=1,\ldots,m$,
\[
\lambda_i(x) = \sum\limits_{j=1}^{m} w_j(x) u_{ij}(x).
\]
The weakly Pareto set $\mathcal{WP}$ is therefore represented as
\begin{equation}\label{rep:wpp5}
\mathcal{WP}=\left\{x\in\r^n\,\begin{array}{|l}
   \sum\limits_{j=1}^{m} w_j(x) \nabla f_j(x) =\sum\limits_{i=1}^{l}\lambda_i(x)\nabla c_i(x) \\
     0\le   \lambda_i(x) \perp  c_i(x) \ge 0,\, i = 1, \ldots, l,  \\
     w(x)\ge 0,\, w_1(x) + \cdots + w_m(x) = 1 \\
\end{array}\right\} .
\end{equation}
The above description is solely in terms of $x$.

\begin{prop}
\label{prop:nonsingular}
    For the MOP \reff{mop} with objectives given by \reff{spe:MOP_new}, the equation \reff{def:P} reduces to
\begin{equation}
\label{Mx}
    \begin{bmatrix}
        \bar{c}_1(x) & \dots & \bar{c}_m(x) \\
        v_{11}(x) c_1(x) & \dots & v_{1m}(x) c_1(x)\\
        \vdots & \vdots & \vdots \\
  v_{l1}(x) c_l(x)  & \dots & v_{lm}(x) c_l(x)\\
  1 & \dots & 1\\
\end{bmatrix}
w = \begin{bmatrix}
        \nabla h(x)-\sum\limits_{i=1}^l h_i(x)\nabla c_i(x)\\
         -h_1(x) c_1(x)  \\   \vdots \\   -h_l(x) c_l(x)\\1
    \end{bmatrix},
\end{equation}
where
\[
v_{ij}(x)   =  \ell_i(x)^Td_j,\, h_i(x)=\ell_i(x)^T\nabla h(x),
\]
\[
\bar{c}_j(x)=\sum\limits_{i=1}^l v_{ij}(x) \nabla c_i(x)-d_j.
\]
For the case $K=\mR^n$, if $D$ has full column rank,
then $w(x)$ can be expressed as
\[
w(x) \, =  \, (D^TD)^{-1}D^T\hat{h}_1(x).
\]
 Here, $D$ and $\hat{h}_1(x)$ are as in \reff{sec3:def:D2}.
\end{prop}
\begin{proof}
One can easily verify that \reff{def:P} reduces to \reff{Mx},
and we omit it for neatness. When $K=\mR^n$,
there are no constraints, so \reff{Mx} becomes
\[
D  w  =  \hat{h}_1(x),
\]
which implies the above formula for $w(x)$.
\end{proof}

\begin{exa}
\label{example P nonsingular}
Consider the MOP \reff{mop} with two objective functions:
\[
f_1(x)=x_1^2+x_2^2+x_1-2x_2,\, f_2(x)=x_1^2+x_2^2+2x_1-2x_2
\]
over the feasible set
\[
K=\{x\in\r^2: 1-x_1^2-x_2\ge 0\}.
\]
Then, we have
\[
C(x)=\begin{bmatrix}
    -2x_1\\-1\\1-x_1^2-x_2
\end{bmatrix},\quad L_1(x) = \begin{bmatrix} 0 & -1
\end{bmatrix},
\]
and the Lagrange multiplier expression
\[
\lambda(x,w)= L_1(x)(w_1\nabla f_1(x)+w_2\nabla f_2(x))=2-2w_1x_1-2w_2x_2.
\]
The matrix $P(x)$ as in \reff{def:P} is
\[
P(x) = \begin{bmatrix}
    -1 - 4x_1 & -2-4x_1\\
    0 & 0\\
    -2x_1^2-2x_2+2&-2x_1^2-2x_2+2\\1&1
\end{bmatrix}.
\]
One can verify that  $P(x)$ is nonsingular, and $P^{\prime}P=I_{2}$ for
\[
P^{\prime}(x)=\begin{bmatrix}
    1 & 0 & 0 & -4x_1x_2+6x_1+2\\
    -1 & 0 & 0 & 4x_1x_2-6x_1-1
\end{bmatrix}.
\]
Hence, the vector of polynomials
\[
w(x) = P'(x) e_{n+l+1}=\begin{bmatrix}
-4x_1x_2+6x_1+2\\ 4x_1x_2-6x_1-1\end{bmatrix}
\]
is the polynomial expression for the weight vector, and the vector of polynomials
\[
\lambda(x)=8x_1^2x_2-8x_1x_2^2-12x_1^2+12x_1x_2-4x_1+2x_2+2
\]
is the polynomial expression for the multiplier vector.

\begin{comment}
Consider the MOP \reff{mop} with two objective functions:
\[
f_1(x)=x_1-2x_2,\, f_2(x)=2x_1-2x_2
\]
over the feasible set
\[
K=\{x\in\r^2: 1-x_1^2-x_2\ge 0\}.
\]
Then, we have
\[
C(x)=\begin{bmatrix}
    -2x_1\\-1\\1-x_1^2-x_2
\end{bmatrix},\quad L_1(x) = \begin{bmatrix} 0 & -1
\end{bmatrix},
\]
and the Lagrange multiplier expression
\[
\lambda(x,w)= L_1(x)(w_1\nabla f_1(x)+w_2\nabla f_2(x))=2w_1+2w_2.
\]
The matrix $P(x)$ as in \reff{def:P} is
\[
P(x) = \begin{bmatrix}
    -1 - 4x_1 & -2-4x_1\\
    0 & 0\\
    -2x_1^2-2x_2+2&-2x_1^2-2x_2+2\\1&1
\end{bmatrix}.
\]
One can verify that  $P(x)$ is nonsingular, and $P^{\prime}P=I_{2}$ for
\[
P^{\prime}(x)=\begin{bmatrix}
    1 & 0 & 0 & 2+4x_1\\
    -1 & 0 & 0 & -1-4x_1
\end{bmatrix}.
\]
Hence, the vector of polynomials
\[
w(x) = P'(x) e_{n+l+1}=\begin{bmatrix}
2+4x_1\\ -1-4x_1\end{bmatrix}
\]
is the polynomial expression for the weight vector, and the vector of polynomials
\[
\lambda(x)=2w_1+2w_2=2
\]
is the polynomial expression for the multiplier vector.
\end{comment}
\end{exa}

\subsection{Comparisons of representations \reff{rep:wpp}, \reff{rep:wpp2} and \reff{rep:wpp5}}
\label{section: comparison}
We make some comparisons between the three representations of
the weakly Pareto set $\mathcal{WP}$ mentioned above.

\bit
\item[(i)]
When the matrix  $C(x)$ as in \reff{sec3:def:C} is nonsingular, we can construct the representation \reff{rep:wpp}, which expresses $\mathcal{WP}$ in terms of $x,w$. This is based on the expression \reff{ref:L1}, where the Lagrange multiplier vector $\lambda$ is expressed in terms of  $x,w$.
When the number of objectives $m$ is relatively small and the number of constraints
$\ell$ is relatively large, this representation efficiently reduces the computational cost.

\item[(ii)]
When the matrix $Q(x)$ as in  \reff{sec3:def:D} is nonsingular, we can express the weight vector $w$
in terms of $x$ and $\lambda$,  leading to   the representation \reff{rep:wpp2}. This is particularly beneficial when the number of constraints  $\ell$ is relatively small and the number of objectives $m$ is relatively large,
as it helps to reduce computational complexity.

\item[(iii)]
When both the matrix  $C(x)$ and the matrix $P(x)$ as in \reff{def:P} are nonsingular,
we can eliminate both the Lagrange multiplier vector $\lambda$ and the weight vector $w$,
obtaining the representation \reff{rep:wpp5}. If the matrices $C^{\prime}(x)$ and $P^{\prime}(x)$
have low degrees, this representation is computationally efficient and convenient.

\eit

%However, using variable $w$ will increase the dimension of variables by $m-1$ especially when the MOP consists of a large number of objective functions, and thus make the computation and storage more expensive.

%When the number of constraints is less than the number of objective functions in MOP, the representation \reff{rep:wpp2} can reduce the computational cost with the assumption that the matrix $Q(x)$ as in equation \reff{sec3:def:D} is nonsingular. It is also worthy noting that the computational cost of using the representation \reff{rep:wpp2} will increase when $Q(x)$ contains high-degree terms.

%With the assumption that the matrix $P(x)$ as in \reff{def:P} is nonsingular, using the representation \reff{rep:wpp5} can minimize the number of variables, significantly reducing both computational and storage demands. Though it is generally difficult to check the nonsingularity of $P(x)$, \reff{rep:wpp5} can be easily applied when the objective functions have special forms such as linear functions or in the form \reff{spe:MOP}. We refer to Section \ref{section5} for the comparison between the run time of solving randomly generated OWPs with different representations.

\section{Moment-SOS Relaxations for the OWP}
\label{section4}

In this section, we apply Moment-SOS relaxations to solve the OWP \reff{oop},
using representations of the weakly Pareto set $\mc{WP}$ given in Section \ref{section3}. We refer to Section \ref{section: comparison} for differences  between different representations and algorithms.

\subsection{The Moment--SOS hierarchy based on the representation \reff{rep:wpp}}

When the MOP \reff{mop} is convex and the polynomial matrix
$C(x)$ as in \reff{sec3:def:C} is nonsingular,
the weakly Pareto set $\mathcal{WP}$  can be represented as in \reff{rep:wpp}.
Then, the OWP \reff{oop} is equivalent to
\be  \label{equ1:oop}
\left\{ \baray{rl}
\min\limits_{(x,w)} & f_0(x)  \\
\st & \sum\limits_{j=1}^{m} w_j \nabla f_j(x) =
\sum\limits_{i=1}^l \lambda_i(x,w)\nabla c_i(x),\\
%\lambda_1(x,w)\nabla c_1(x)+\dots +\lambda_l(x,w)\nabla c_l(x), \\
&\lambda_i(x,w)\cdot c_i(x)=0, \, i =1, \ldots, l, \\
&\lambda_i(x,w)\geq 0,\, c_i(x)\ge 0, \, i =1, \ldots, l, \\
&1-  \|w\|^2\ge 0,\, e^Tw-1=0,
\\
& w = ( w_1,\dots,w_{m})\ge 0,\, %\sum\limits_{j=1}^m w_j-1 = 0,
\earay \right.
\ee
which is a polynomial optimization in $(x,w)$. We note that the additional constraint
$1-\sum_{i=1}^mw_i^2\ge 0$ does not affect the feasible set of \reff{equ1:oop}, but it can effectively improve the numerical performance of our subsequent algorithms.
In particular, if \reff{equ1:oop} is infeasible, then there are no WPPs.
Denote the polynomial tuples
\begin{equation}\label{def:con}
\begin{split}
\Phi_{x,w}:=& \Big \{\sum\limits_{j=1}^{m} w_j \nabla_{x_k} f_j(x)-
     \sum\limits_{i=1}^{l} \lambda_i(x,w)\nabla_{x_k} c_i(x) \Big \}_{k \in [n] } \cup \\
    & \big \{\lambda_i(x,w)\cdot c_i(x) \big \}_{ i \in [l] }   \cup \{e^T w-1\},  \\
\Psi_{x,w}:= & \big\{1-\|w\|^2\big\}\cup\big \{\lambda_i(x,w)  \big \}_{ i\in [l] }  \cup
       \big \{c_i(x) \big \}_{ i\in [l] } \cup \big\{w_i \big \}_{ i \in [m] }.
\end{split}
\end{equation}

Denote the degree
\[
d_0:= \max \{\lceil\operatorname{deg}(p)/2\rceil,\,p\in\Phi_{x,w}\cup\Psi_{x,w}\}.
\]
The problem \reff{equ1:oop} can be rewritten as
\be \label{simplep}
\left\{
\begin{array}{cl}
    \min\limits_{(x,w)} & f_0(x) \\
    \st & \phi(x,w)=0\,(\forall\, \phi \in\Phi_{x,w}),\\
    & \psi(x,w)\ge 0\,(\forall\, \psi\in \Psi_{x,w}).
\end{array}   \right.
\ee
For a degree $k\ge d_0$, the $k$th order SOS relaxation for solving \reff{simplep} is
\begin{equation}
\label{sos_k}
    \left\{
    \begin{array}{cl}
        \max\limits_{\gamma}  & \gamma \\
        \st & f_0 - \gamma \in \ideal{\Phi_{x,w}}_{2k} + \qmod{\Psi_{x,w}}_{2k}.
    \end{array}
    \right.
\end{equation}
The dual optimization problem of \reff{sos_k} is the $k$th order moment relaxation:
\begin{equation}
\label{mom_k}
\left\{
\begin{array}{cl}
    \min\limits_{y} & \langle f_0,y\rangle \\
    \st
     & L_{\phi}^{(k)}[y]=0 \, (\phi\in\Phi_{x,w}),   \\
    & L_{\psi}^{(k)}[y]\succeq 0\, (\psi\in\Psi_{x,w}),  \\
    & M_k[y]\succeq 0,\\
    &  y_0=1, \,y \in \mR^{n+m}_{2k}.\\
\end{array}
\right.
\end{equation}
Denote the optimal values of \reff{sos_k} and \reff{mom_k} by $f_{sos,k}$, $f_{mom,k}$,
respectively.  The hierarchy of relaxations \reff{sos_k}--\reff{mom_k} is said to
have finite convergence if $f_{sos,k}=f_{min}$ for all $k$ sufficiently large.

In practice, the flat  truncation  condition (see \cite{niecertificate,HenLas05})
is often used to detect finite convergence and to extract minimizers.
Suppose $y^{*}$ is a minimizer of $\reff{mom_k}$ for a relaxation order $k$.
If there exists an integer $t \in[d_0, k]$ such that
\be \label{rank}
\operatorname{rank} M_{t} [ y^* ] \, = \,
\operatorname{rank} M_{t-d_0}[ y^* ],
\ee
then
the truncation $y^*|_{2t}$ admits a finitely $r$-atomic probability measure $\mu^*$
whose support is contained in $K$.
That is, there exist points $(u_1,w_1),\ldots, (u_r,w_r)\subseteq \mR^{n+m}$,
which are feasible points of \reff{equ1:oop}, such that
\[
\mu^*=\sum_{j=1}^r \gamma_j \delta_{(u_j,w_j)}, \quad
\sum_{j=1}^r \gamma_j=1, \quad \gamma_j>0, \quad j=1, \ldots, r,
\]
where $\delta_{(u_j,w_j)}$ denotes the unit Dirac measure supported at $(u_j,w_j)$.
One can further show that $f_{k,mom}=f_{\min}$ and $(u_1,w_1),\ldots, (u_r,w_r)$
are minimizers of \reff{equ1:oop}.

The following is the algorithm for solving \reff{equ1:oop}.

\begin{alg} \label{alg1}
Let $\Phi_{x,w},\Psi_{x,w}$ be as in \reff{def:con} and   let $k := d_0$.

\begin{description}

\item[Step ~1]  Solve the moment relaxation \reff{mom_k}.
If it is infeasible, then \reff{oop} is infeasible
(i.e., there are no weakly Pareto points) and stop; otherwise,
solve \reff{mom_k} for a minimizer $y^{\ast}$.

\item[Step~ 2] Let $t:=d_0$. If $y^{\ast}$ satisfies the rank condition \reff{rank},
then extract $r \coloneqq \rank M_t[y^{\ast}]$ minimizers for \reff{equ1:oop} and stop.

\item[Step 3~] If \reff{rank} fails to hold and $t<k$, let $t:=t+1$ and go to Step~2;
otherwise, let $k = k+1$ and go to Step~1.

\end{description}
\end{alg}

The following is the convergence result for Algorithm~\ref{alg1}.

\begin{thm}
Suppose the MOP \reff{mop} is convex, the matrix $C(x)$
%polynomial tuple $(c_1,\dots,c_l)$
is nonsingular.
%and the quadratic module $\ideal{\Phi_{x,w}} + \qmod{\Psi_{x,w}}$ is Archimedean.
Then, we have:

\bit

\item[(i)] If the moment relaxation \reff{mom_k} is infeasible,
the weakly Pareto set $\mathcal{WP}=\emptyset$.
Conversely, if the set $\mathcal{WP}=\emptyset$, then the relaxation \reff{mom_k}
is infeasible for sufficiently large $k$.

\item[(ii)]
Suppose that the set $\mathcal{WP}\neq \emptyset$
and  the quadratic module $\ideal{\Phi_{x,w}} + \qmod{\Psi_{x,w}}$ is Archimedean. Then, we have $f_{k,sos}\rightarrow f_{\min}$.
Furthermore, if $\Phi_{x,w}(x)=0$ has only finitely many real solutions,
then \reff{rank} holds when $k$ is sufficiently large.

\eit

\end{thm}
\begin{proof}
(i)  Suppose the weakly Pareto set $\mathcal{WP}\neq \emptyset$,
and let $x^*\in \mathcal{WP}$. By Theorem \ref{theorem1},  $x^*$ is a minimizer of
 the linear scalarization problem \reff{lsp}
for some weight vector $w\in \Delta^{m-1}$.
Then, we know that $(x^*,w)$ is feasible for \reff{equ1:oop}, and the truncated multisequence   $[(x^*,w)]_{2k}$ is feasible for \reff{mom_k}.
If $\mathcal{WP}=\emptyset$, we have
\[
\left \{(x,w)\in \mR^{n+m}
\left| \baray{l}
\phi(x,w)=0\,(\forall\, \phi \in\Phi_{x,w}),\\
\psi(x,w)\geq 0\,(\forall\, \psi \in\Psi_{x,w})
\earay \right.
\right \}=\emptyset.
\]
By Positivstellensatz \cite{nie2023moment}, there exist $h\in \ideal{\Phi_{x,w}}$,
$s\in \operatorname{Pre}[\Psi_{x,w}]$ such that $2+h+s=0$, where $\operatorname{Pre}[\Psi_{x,w}]$ is the preordering generated by the polynomial tuple $\Psi_{x,w}$. Since $1+s(x)>0$ for $x\in \mathcal{WP}$ and the set $\ideal{\Phi_{x,w}} + \qmod{\Psi_{x,w}}$ is Archimedean, we have  $1+s\in \ideal{\Phi_{x,w}} + \qmod{\Psi_{x,w}}$. It implies that
\[
-1=1+h+s\in \ideal{\Phi_{x,w}} + \qmod{\Psi_{x,w}}.
\]
This implies that
\begin{eqnarray*}
f_0 - \gamma & = &  \frac{(f_0+1)^2}{2}+(-1)\cdot \frac{(f_0-1)^2}{2}+(-1)\cdot \gamma  \\
  & \in  &  \ideal{\Phi_{x,w}}_{2k} + \qmod{\Psi_{x,w}}_{2k}
\end{eqnarray*}
for all $\gamma\geq 0$ when $k$ big enough. Then, we know that \reff{sos_k}
is unbounded above and \reff{mom_k} is infeasible for $k$ sufficiently large.

(ii) Since the set $\ideal{\Phi_{x,w}} + \qmod{\Psi_{x,w}}$ is Archimedean,
the asymptotic convergence  $f_{k,mom}\rightarrow f_{\min}$
follows from \cite{158}. When $\Phi_{x,w}(x)=0$
has only finitely many real solutions,
we refer to \cite{laurentfinite, niefinite} for this conclusion.
\end{proof}

\begin{comment}
\begin{remark}
The equality constraint $\sum\limits_{j=1}^m w_j  = 1$ implies
\[
w_m = 1- (w_1 + \cdots + w_{m-1} ).
\]
Thus, we can further eliminate the variable $w_m$ in \reff{equ1:oop},
to save the computational expense.
\end{remark}
\end{comment}

\subsection{The Moment--SOS hierarchy based on the representation \reff{rep:wpp2}}

When the MOP \reff{mop} is convex and the polynomial matrix $Q(x)$
as in \reff{sec3:def:D} is nonsingular,
the weakly Pareto set  $\mathcal{WP}$ can be represented as in \reff{rep:wpp2}.
Then, the OWP \reff{oop} is equivalent to
\be  \label{ref4.2}
\left\{ \baray{rl}
\min\limits_{(x,\lambda)} & f_0(x)  \\
\st & \sum\limits_{j=1}^{m} w_j(x,\lambda) \nabla f_j(x)
 =\sum\limits_{i=1}^{l}\lambda_i\nabla c_i(x), \\
&\lambda_i\cdot c_i(x)=0, \, i =1, \ldots, l, \\
&\lambda_i\geq 0,\, c_i(x)\ge 0, \, i =1, \ldots, l, \\
& 1-\|w(x,\lambda)\|^2\ge0,\,
1-e^Tw(x,\lambda)=0,\\
%1-\sum\limits_{i=1}^m w_i(x,\lambda)^2\ge 0,\, 1-\sum\limits_{i=1}^m w_i(x,\lambda)=0,
& w = ( w_1,\dots,w_{m})\ge 0,
%w_1(x,\lambda)\geq 0,\dots, w_m(x,\lambda)\geq 0,\\
% w_1(x,\lambda)+\cdots+w_m(x,\lambda)-1=0,
\earay \right.
\ee
This is a polynomial optimization problem in $(x,\lambda)$.
Denote the polynomial tuples
\be  \label{def:con*}
\begin{split}
\Phi_{x,\lambda}:=& \Big\{\sum\limits_{j=1}^{m} w_j(x,\lambda) \nabla_{x_k} f_j(x) -\sum\limits_{i=1}^{l}\lambda_i \nabla_{x_k}c_i(x) \Big\}_{ k \in [n] }  \cup \\
& \big  \{\lambda_i\cdot c_i(x) \big \}_{ i \in [l] } \cup
  \big  \{
  %\sum\limits_{j=1}^m w_j(x,\lambda)-1 \big
  1-e^Tw(x,\lambda)
  \},
%\{w_1(x,\lambda)+\cdots+w_m(x,\lambda)-1\},
\\
\Psi_{x,\lambda}:= & \big\{1-\|w(x,\lambda)\|^2\}%\sum\limits_{i=1}^m w_i(x,\lambda)^2\big\}\cup\big \{\lambda_i  \big \}_{ i \in [l] }
\cup
    \big \{c_i(x) \big \}_{ i\in [l] }  \cup \big \{w_i(x,\lambda) \big \}_{ i\in [l] }.
\end{split}
\ee
Denote the degree
\[
d_0^*  \coloneqq  \max \{\lceil\operatorname{deg}(p)/2\rceil, \,
 p\in\Phi_{x,\lambda}\cup\Psi_{x,\lambda}\}.
\]
For a degree $k\geq d_0^*$, the $k$th order SOS relaxation for solving \reff{ref4.2} is
\begin{equation}
\label{sos_k*}
\left\{
\begin{array}{cl}
    \max\limits_{\gamma}  & \gamma \\
    \st & f_0 - \gamma \in \ideal{\Phi_{x,\lambda}}_{2k} + \qmod{\Psi_{x,\lambda}}_{2k}.
\end{array}
\right.
\end{equation}
The dual optimization problem of \reff{sos_k*} is the $k$th order moment relaxation:
\begin{equation}
\label{mom_k*}
\left\{
\begin{array}{cl}
    \min\limits_{y} & \langle f_0,y\rangle \\
    \st
    & L_{\phi}^{(k)}[y]=0 \, (\phi\in\Phi_{x,\lambda}),  \\
    & L_{\psi}^{(k)}[y]\succeq 0\, (\psi\in\Psi_{x,\lambda}),  \\
    & M_k[y]\succeq 0,\\
    & y_0=1,\,y\in \mR^{n+l}_{2k}.
\end{array}
\right.
\end{equation}

The following algorithm is the analogue of Algorithm \ref{alg1} for solving \reff{ref4.2}.

\begin{alg} \label{alg2}
Let $\Phi_{x,\lambda},\Psi_{x,\lambda}$ be as in \reff{def:con*} and   let $k := d_0$.

\begin{description}

\item[Step ~1]  Solve the semidefinite relaxation \reff{mom_k*}.
If it is infeasible, then \reff{oop} is feasible and stop;
otherwise, solve it for a minimizer $y^{\ast}$.

\item[Step~ 2] Let $t:=d_0^*$. If $y^{\ast}$ satisfies the rank condition \reff{rank},
then extract  $r:=\rank M_t[y^{\ast}]$ minimizers for \reff{ref4.2} and stop.

\item[Step 3~] If \reff{rank} fails to hold and $t<k$, let $t:=t+1$ and go to Step 2; otherwise,
let $k = k+1$ and go to Step 1.

\end{description}
\end{alg}

The convergence of Algorithm \ref{alg2} is similar to that of Algorithm \ref{alg1}.
We omit it for the cleanness of the paper.

\subsection{The Moment--SOS hierarchy based on the representation \reff{rep:wpp5}}

When the MOP \reff{mop} is convex and the polynomial matrices $C(x)$, $P(x)$ are nonsingular,
the weakly Pareto set  $\mathcal{WP}$ can be represented as in \reff{rep:wpp5}.
Then,  the OWP \reff{oop} is equivalent to
\be  \label{ref4.3}
\left\{ \baray{rl}
\min\limits_{x} & f_0(x)  \\
\st &  \sum\limits_{j=1}^{m} w_j(x) \nabla f_j(x) =
   \sum\limits_{i=1}^{l} \lambda_i(x)\nabla c_i(x), \\
&\lambda_i(x)\cdot c_i(x)=0, \, i = 1, \ldots, l, \\
&\lambda_i(x)\geq 0,\, c_i(x)\ge 0, \, i = 1, \ldots, l, \\
& 1-\|w(x)\|^2\ge 0,\, e^Tw(x)-1=0,\\
& w(x) = ( w_1(x),\dots,w_m(x)) \ge 0,\,
%w(x):= ( w_1(x),\dots,w_m(x))\ge 0,\,
%& w_1(x)\geq 0,\dots, w_m(x)\geq 0, \\
%& w_1(x)+\cdots+w_m(x)-1=0,
\earay \right.
\ee
This is a polynomial optimization problem in $x$.
Denote the polynomial tuples
\be \label{def:con'}
\begin{split}
\Phi_x:=& \Big \{\sum\limits_{j=1}^{m} w_j(x) \nabla_{x_k} f_j(x)  -
     \sum\limits_{i=1}^{l} \lambda_i(x) \nabla_{x_k} c_i(x) \Big \}_{ k \in [n] }  \\
&     \cup \big \{\lambda_i(x)\cdot c_i(x) \big \}_{ i\in [l] }
  \cup \Big \{ e^Tw(x)-1 \Big \},
%w_1(x)+\cdots+w_m(x)-1\},
\\
\Psi_x:= & \big\{1-\|w(x)\|^2\big\}\cup\big \{\lambda_i(x) \big \}_{ i\in [l] } \cup \big \{c_i(x) \big \}_{ i\in [l] }
      \cup \big \{w_i(x) \big \}_{ i\in [l] }.
\end{split}
\ee
Denote the degree
\[
d_0' \coloneqq \max \{\lceil\operatorname{deg}(p)/2\rceil,
    p\in\Phi_x\cup\Psi_x\}.
\]
For a degree $k\geq d_0'$, the $k$th order SOS relaxation for solving \reff{ref4.3} is
\be  \label{sos_k'}
    \left\{
    \begin{array}{cl}
        \max\limits_{\gamma}  & \gamma \\
        \st & f_0 - \gamma \in \ideal{\Phi_x}_{2k} + \qmod{\Psi_x}_{2k}.
    \end{array}  \right.
\ee
The dual optimization problem of \reff{sos_k'} is the $k$th order moment relaxation:
\begin{equation}
\label{mom_k'}
\left\{
\begin{array}{cl}
    \min\limits_{y} & \langle f_0,y\rangle \\
    \st
    &L_{\phi}^{(k)}[y]=0 \, (\phi\in\Phi_x),  \\
    & L_{\psi}^{(k)}[y]\succeq 0\, (\psi\in\Psi_x),  \\
    & M_k[y]\succeq 0, \\
    & y_0=1,\,y\in \mR^{n}_{2k}.
\end{array}
\right.
\end{equation}

The following algorithm is analogous of Algorithm \ref{alg1}
for solving \reff{ref4.3}.

\begin{alg} \label{alg3}
Let $\Phi_x,\Psi_x$ be as in \reff{def:con'} and   let $k := d_0'$.

\begin{description}

\item[Step ~1]  Solve the semidefinite relaxation \reff{mom_k'}. If it is infeasible, then \reff{oop} is feasible and stop; otherwise, solve it for a minimizer $y^{\ast}$.

\item[Step~ 2] Let $t:=d_0'$. If $y^{\ast}$ satisfies the rank condition \reff{rank},
then extract  $r:=\rank M_t[y^{\ast}]$ minimizers for \reff{ref4.3} and  stop.

\item[Step 3~] If \reff{rank} fails to hold and $t<k$, let $t:=t+1$ and go to Step 2; otherwise, let $k = k+1$ and go to Step 1.

\end{description}
\end{alg}

The convergence of Algorithm \ref{alg3} is similar to that of Algorithm \ref{alg1}. 
We omit it for cleanness of the paper.

\section{Numerical experiments}
\label{section5}

In this section, we apply Algorithms  \ref{alg1}, \ref{alg2}, \ref{alg3} to solve polynomial optimization over the weakly Pareto sets $\mathcal{WP}$ of convex MOPs. For a given MOP, the choice of representation for $\mc{WP}$ and the algorithm depends on the problem structure and the number of objectives and constraints. We refer to Subsection~\ref{section: comparison} for how to select an appropriate representation.
The Moment-SOS relaxations   \reff{sos_k}--\reff{mom_k}, \reff{sos_k*}--\reff{mom_k*}, \reff{sos_k'}--\reff{mom_k'}  are solved by the software
{\tt GloptiPoly~3} \cite{gloptipoly},
which calls the SDP solver {\tt SeDuMi} \cite{sedumi}. The computation is implemented in MATLAB R2023b on a laptop with  16G RAM.
To maintain the neatness of the paper, the computational results are presented with four decimal digits.

\subsection{Examples using the representation~\reff{rep:wpp} }
This subsection gives  some examples of applying Algorithm \ref{alg1} to solve the
OWP  \reff{oop}.
In these examples, the number of constraints is relatively large, making the use of representation \reff{rep:wpp} and Algorithm \ref{alg1} more efficient.

\begin{exa}
\label{exa1}
Consider the OWP with the
preference function
\[
f_0(x) = (x_1-x_2+x_7^2-x_8)^2-(x_3-x_4-2x_5-4x_6)^3,
\]
and the objective function $F(x)=(f_1(x),f_2(x))$, where
\[
f_1(x) =  \sum\limits_{i=1}^8 x_i^2 - \sum\limits_{i=1}^8 x_i,~f_2(x) = (x_1-x_2)^2+(x_3-x_4)^2+(x_5-x_6)^2.
\]
The feasible set is given by
$$
K = \{x\in\r^8:\, 1-x_i^2\ge 0\,\,\, (i=1,\dots,8)\}.
$$
We use the polynomial expression as in  Example \ref{cube ref} (i) and  the Lagrange multiplier vector
$\lambda(x,w)= (\lambda_1(x,w),\dots,\lambda_8(x,w))$ can be represented as
$$
\lambda_i(x,w) = -\frac{1}{2}x_i( w_1 \frac{\partial f_1}{\partial x_i} + w_2 \frac{\partial f_2}{\partial x_i}),\, i=1,\dots,8.
$$
By Algorithm \ref{alg1}, we have $f_{min}=-216.0000$ and obtain that
$
w=(0.0000,1.0000)$,
$$
\,x=(0.0013,0.0013,0.0012,0.0012,-1.0000,-1.0000,0.0000,0.0000),
$$
at the relaxation order $k = 3$. The computation takes around 95.03 seconds.
\end{exa}

\begin{exa}
Consider the OWP with the preference function
\[
f_0(x)=(x_1^2-3)(x_2+1)-3x_3x_4-x_5^2x_6
\]
and the objective function $F(x)=(f_1(x),f_2(x))$, where
\[
f_1(x)=(x_1+2x_2)^2+(x_3+3x_4)^2+x_5,
~
f_2(x) = x_1+x_2+(x_3-\frac{1}{2}x_4+x_6)^2.
\]
The feasible set is a polyhedron given as
\begin{comment}
\[
K = \left\{
x\in\r^6:\,Cx\ge 0
\right\},
\]
where
$$
C=\begin{bmatrix}
\frac{1}{2}&-1&3&1&1&-1\\2&-\frac{1}{2}&0&5&2&3\\-2&-1&-4&3&6&-\frac{7}{3}\\
-\frac{9}{4}&-\frac{5}{2}&-1&2&2&\frac{8}{3}\\
2&-\frac{8}{3}&4&0&\frac{5}{2}&-5
\end{bmatrix}
$$
\end{comment}
\[
 K = \left\{x\in\r^{6}\,\begin{array}{|l}
        -\frac{1}{2}x_1-x_2+3x_3+x_4+x_5-x_6\ge 0, \\ 2x_1-\frac{1}{2}x_2+5x_4+2x_5+3x_6\ge 0, \\
         -2x_1-x_2-4x_3+3x_4+6x_5-\frac{7}{3}x_6\ge 0,\\-\frac{9}{4}x_1-\frac{5}{2}x_2-x_3+2x_4+2x_5+\frac{8}{3}x_6\ge 0,\\
         2x_1-\frac{8}{3}x_2+4x_3+\frac{5}{2}x_5-5x_6\ge 0
    \end{array}\right\}.
\]
Then, we use the polynomial expression as in Example \ref{cube ref} (ii) and  the Lagrange multiplier vector $\lambda(x,w)=C^{\prime}_1 (w_1 \nabla f_1(x)+w_2\nabla f_2(x))$ for
   $$
   C^{\prime}_1(x) = (CC^T)^{-1}C.
   $$
    By Algorithm \ref{alg1}, we get $f_{min}=   -7.5140$ and obtain that $w=(  0.9998,0.0002)$,
    \[
    x = ( -0.0058,
   -0.0623,
   -0.0333,
    0.0672,
   -0.0907,
   -0.0581),
    \]
    at the relaxation order $k = 2$. The computation takes around 11.42 seconds.
\end{exa}

\begin{comment}
\begin{exa}
\label{exa3}
Consider the OWP with the preference function
    $$
    f_0(x) = \sum\limits_{i=1}^{10}x_i,
    $$
 and the objective function $F(x)=(f_1(x),f_2(x))$, where
    $$
     f_1(x) = \sum\limits_{i=1}^{10}x_i^2, ~ f_2(x) = 2x_1^2+x_2^2+x_3^2+x_6^2-x_5-x_7-x_8-x_9-x_{10}.
    $$
   The feasible set is a polyhedron given as
    $$
    K = \left\{x\in\r^{10}\,\begin{array}{|l}
         x_1+x_2+3x_8-1\ge 0, \\  x_3+3x_4-x_9\ge 0, \\
         x_4+2x_6-x_7\ge 0,\\2x_1+3x_6-3x_7\ge 0,\\
         3x_2+5x_4+x_9+3x_{10}\ge 0
    \end{array}\right\}.
    $$
Then, we use the polynomial expression as in  Example \ref{cube ref} (ii) and  the Lagrange multiplier vector $\lambda(x,w)=C^{\prime}_1 (w_1 \nabla f_1(x)+w_2\nabla f_2(x))$ for
   $$
   C^{\prime}_1(x) = (CC^T)^{-1}C.
   $$
By Algorithm \ref{alg1},  we get $f_{min} =  0.0000$ and obtain that $w=(1.0000,0.0000)$,
\begin{small}
$$
x=(0.0788,
    0.1694,
   -0.1560,
    0.4071,
   -0.2626,
    0.1294,
   -0.4418,
    0.5449,
   -0.3285,
   -0.1407),
$$
\end{small}
 at the relaxation order $k = 2$.
The computation takes around 6.16 seconds.
\end{exa}
\end{comment}

\begin{exa}
Consider the OWP with the preference function
\[
f_0(x)= x_3x_5x_8-x_1^2x_2+x_4+x_6+x_7
\]
and objective function $F(x)=(f_1(x),f_2(x)),$ where
\[
f_1(x) = (x_1+x_3)^2- \sum\limits_{i=1}^8 x_i, ~ f_2(x)=(x_7+x_8)^2+2 \sum\limits_{i=1}^8 x_i.
\]
The feasible set is given by
\[
K = \left\{x\in\r^{8}\,
\begin{array}{|l}
2x_1-x_2^2 -x_3^2+x_5\ge 0,\\
x_2-x_3^2-x_6^2\ge 0,\\
x_3-x_5^2+x_8\ge 0,\\
4x_4+x_6-x_7^2\ge 0,\\
-x_5\ge 0
\\
\end{array}
\right\}.
\]
Then, we use the polynomial expression as in  Example \ref{cube ref} (iii), and  the Lagrange multiplier vector    $ \lambda(x,w) = T(x)^{-1}(w_1\nabla f_1(x)+w_2\nabla f_2(x))$, for
$$
T(x)^{-1}=\begin{bmatrix}
    \frac{1}{2}& 0 &0&0&0\\x_2&1&0&0&0\\x_3+2x_2x_3&2x_3&1&0&0\\0&0&0&\frac{1}{4}&0\\\frac{1}{2}(1-4x_3x_5-8x_2x_3x_5)&-4x_3x_5&-2x_5&0&-1
\end{bmatrix}.
$$
By Algorithm \ref{alg1}, we get $f_{min}=   -2.1361$ and obtain that
\begin{small}
$$
w = (0.6667,0.3333),\,x=(   0.9349  ,  0.9979 ,  -0.9349 ,   1.0878 ,  -0.0000  , -0.3519,   -1.9999  ,  1.9999)
$$
\end{small}
at the relaxation order $k=3$. The computation takes around 167.03 seconds.

\end{exa}

\begin{comment}
\begin{exa}
Consider the OWP with the preference function
$$
f_0(x) = x_1-x_2x_3,
$$
and the objective function $F(x)=(f_1(x),f_2(x), f_3(x))$, where
$$
f_1(x) = x_1^2x_2^2+x_2^2x_3^2, ~f_2(x) = 2x_1^4+x_2^4+3x_3^4, ~f_3 (x)= x_1^2+3x_2^2.
$$
The feasible set is given by
$$
K=\{x\in\r^3: x_1-x_2^2\ge 0,\,2x_2-x_3^2\ge 0,\,x_3\ge 0\}.
$$
Then, we use the polynomial expression as in  Example \ref{cube ref} (iii), and  the Lagrange multiplier vector    $ \lambda(x,w) = T(x)^{-1}(\sum\limits_{j=1}^3 w_j \nabla f_j(x))$, for
$$
T^{-1}(x)=\begin{bmatrix}
    1&0&0\\x_2&\frac{1}{2}&0\\2x_2x_3&x_3&1
\end{bmatrix}.
$$
By Algorithm \ref{alg1}, we get $f_{\min}=0$ and obtain that
$$
(w_1,w_2,w_3) =(   0.4272, 0.3763, 0.1965),\,\, (x_1,x_2,x_3) = (0, 0.0439, 0.0082),
$$
 at relaxation order $k=3$. The computation takes around 13.45 seconds.

\end{exa}
\end{comment}

\subsection{Examples using representation \reff{rep:wpp2} }
This subsection gives  some examples of applying Algorithm \ref{alg2}
to solve the OWP \reff{oop}. In these examples, the number of objective functions is relatively large, making the use of \reff{rep:wpp2} and Algorithm \ref{alg2} more efficient.

\begin{exa}
Consider the OWP with the preference function
\[
f_0(x) = -\sum\limits_{i=1}^8 x_i^4+x_1x_3^2-x_1
\]
and the objective function $F(x)=(f_1(x),\dots,f_6(x))$, where
\[
f_1(x)=h(x)+x_1+2x_2+5x_5,~ f_2(x) = h(x)+2x_2+x_3-x_4, ~ f_3(x) = h(x)+x_7+x_{8},
\]
\[
f_4(x)=h(x)+x_5+x_6+x_7-3x_8,~ f_5(x)=h(x)-3x_1-3x_4, ~ f_6(x)=h(x)+x_3-x_8,
\]
and
$$
h(x)= x_1^2+2x_1x_2+x_1x_3+2x_2^2+2x_3^2+x_5+x_6^2.
$$
The feasible set is given by
\[
K = \{x\in\r^8:\, 1-\|x\|^2\ge 0\}.
\]
Note that the objective functions only differ by the linear terms. The weight vector $w$ can be represented as in \reff{rep:w}.
%with
%$$
%D = \begin{bmatrix}
%1 & 0 & -1 & 0& -3 & 0 \\
%2 & 2 & 2 & 0&0&0 \\
%0 & 1 & 3&0&0&1\\
%0&0&0&0&-3&0\\
%0&0&0&1&0&0\\
%0&0&0&1&0&0\\
%1& 1& 1
%\end{bmatrix},
%$$
%$$
%\hat{h}_1(x)=\begin{bmatrix}
%-2x_1 -2x_2 - x_3 &
%-2x_1 - 4x_2 &
%-x_1 - 4x_3&0&0&-1&-2x_6^2&0&0&1
%\end{bmatrix}^T.
%$$
By Algorithm \ref{alg2},  we get $f_{min} =  -1.0177$ and obtain that
\[
w=(0.0000,0.0000,0.0000,0.0000,1.0000,0.0000),
\]
\[
x = (  0.5677,
   -0.1434,
   -0.0717,
    0.7660,
   -0.2553,
   -0.0000,
   -0.0000,
    0.0000),
\]
%\[
%w=\begin{bmatrix}0&0.2202&0.7798\end{bmatrix}^T,
%\]
at the relaxation order $k=2$.
The computation takes around 7.53 seconds.
\end{exa}

\begin{exa} (random instances)
 We consider the randomly generated unconstrained OWPs.
 Let $Q_0$ and $Q$ be positive definite matrices in $\r^{n\times n},$ and  $d_0,d_1,\dots,d_{n}$ be vectors in $\r^{n}$ such that all entries of them obey the standard normal distribution. The preference function is
\[
f_0(x)=\frac{1}{2}x^TQ_0x + d_0^Tx,
\]
and the objective function $F(x)=(f_1(x),\dots, f_n(x))$, where
\[
f_i(x)=\frac{1}{2}x^TQx+d_i^Tx
\]
 Since the matrix $D$ as in \reff{rep:w} is nonsingular for random $d_1,\dots,d_n$, we can use the representation of the weight vector $w(x)$ given in Proposition \ref{thm33}.

For $n = 5, 10, 20, 50$, and 100, we generate 100 random instances. Algorithm \ref{alg2} is applied to solve them. We also apply the standard Moment-SOS relaxations as in \reff{equ1:oop} to solve  the reformulation without eliminating $w$ as below:
\begin{equation*}
\left\{
\begin{array}{cl}
    \min\limits_{x,w\in\r^n} & \frac{1}{2}x^TQ_0x+d_0^Tx \\
    \text{s.t.} & \sum\limits_{i=1}^n w_i\nabla f_i(x) = 0,\\
    & e^Tw-1=0,\,1-\|w\|^2\ge0,\\
    & w:=(w_1,\dots,w_n)\ge 0.
\end{array}
\right.
\end{equation*}
Both methods successfully solve all instances, and the computation time comparison is reported in Table~\ref{Table:01}.
\begin{table}[htb]
    \centering
    \caption{Computation time (in seconds) for different  values of $n$}
\label{Table:01}

\begin{tabular}{|c|c|c|c|c|c|}  \hline
 $n$ & $5$ & $10$ & $20$ & $50$ & $100$ \\
 \hline
Algorithm \ref{alg2}& 0.04 & 0.07  & 0.10  & 0.81  & 170 \\
\hline
The standard relaxation & 0.11 & 0.17  & 0.33  & 5.53  & $\ge 600$\\
\hline
\end{tabular}
\end{table}

The table shows that Algorithm \ref{alg2} significantly speeds up
the standard Moment-SOS relaxations.
\end{exa}

\subsection{Examples using representation \reff{rep:wpp5} }
This subsection gives an example of applying Algorithm \ref{alg3}
to solve the OWP \reff{oop}.
In these examples, we are able to express the Lagrange multiplier vector $\lambda$ and the weight
vector $w$ in terms of $x$, making the use of \reff{rep:wpp5} and Algorithm \ref{alg3} more efficient.

\begin{exa}
Consider the OWP with the
preference function
\[
f_0(x) = x_1^2x_2+x_2^2x_3-3x_4x_5x_6+x_{10}^2
\]
and the objective functions
\[
f_1(x)=h(x)+3x_1+4x_2+x_5,\, f_2(x)=h(x)-2x_8 -x_9,
\]
\[
f_3(x)=h(x)+2x_{10}-3x_7,\,f_4(x)=h(x)+x_5-x_4-x_3,
\]
for the convex polynomial
\[
h(x)=x_1^2 + 2x_2^2+3x_3^2+\dots + 10x_{10}^2.
\]
The feasible set is given by
\[
    K=\{x\in\r^{10}: 1-2x_1^2-x_3^2+x_5+x_7\ge 0\}.
\]
By Algorithm \ref{alg3},  we get $f_{min} =  -0.4982$ and obtain that
\begin{small}
\[
x = (-0.7058,
   -1.0000,
    0.0000,
    0.0000,
   -0.0437,
   -0.0000,
    0.0402,
    0.0000,
    0.0000,
   -0.0000),
\]
\end{small}
which gives
\[
\lambda =    0.5626, \,w=(1.0000,0.0000,0.0000,0.0000),
\]
at the relaxation order 2. The computation takes around 4.05 seconds.
\end{exa}

\begin{exa}
Consider the OWP with preference function
\[
f_0(x)  =  x^Tx-4(x_1^2x_2^2+x_2^2x_3^2+x_3^2x_4^2+x_5^2x_1^2) 
\]
and the objective functions
\[
f_1(x)=x^Tx+4x_2+x_3+2x_4+3x_5+3x_6+4x_7+3x_{10},
\]
\[
f_2(x)=x^Tx + 3x_1+2x_2+x_3+x_4+2x_5+5x_6+4x_7+2x_8+2x_{10},
\]
\[
f_3(x)=x^Tx + 4x_1+x_2+2x_3+3x_4+x_5+5x_6+3x_7+5x_8+x_9+x_{10}.
\]
The feasible set is given by
\[
K = \{x\in\r^{10}:1-x_1^2-x_2^2-x_3^2-x_4-x_5-x_{10}\ge 0\}.
\]
By Algorithm \ref{alg3}, we get $f_{min}=-0.5221$ and obtain that
\begin{small}
\[
x = -(0.9899,
   1.2470,
   0.6942,
   1.1536,
   1.0263,
   2.0456,
   1.7903,
   1.1734,
   0.2091,
   1.0263),
\]
\end{small}
which gives
\[
\lambda = 0.0167,\, w= (0.4539,0.1277,0.4183)
\]
at the relaxation order 2. The computation takes around 5.68 seconds.
\end{exa}

\begin{comment}
\begin{exa}
Consider the OWP with preference function
$$
f_0(x)=x_1+x_2-x_1^3x_2
$$
and objective functions
\[
f_1(x)=-x_1-x_2+3x_2^2, ~ f_2(x)= 5x_1-x_2+3x_2^2
\]
with feasible set
\[
K = \{x\in\r^2:\, -x_1^4+x_2\ge 0\}.
\]
It can be shown that $L_1(x)=\begin{bmatrix}0 & 1 \end{bmatrix}$ and $\mathcal{WP}$ has the form
\[
\mathcal{WP}=\left\{
x\in\r^2: \, P(x)w(x)=e_4,\,w(x)\ge 0,\,\lambda(x)\ge 0,\,-x_1^4+x_2\ge 0
\right\}
\]
for
\[
P(x)=\begin{bmatrix}
1 + 4 x_1^3 - 24 x_1^3 x_2 & -5 + 4 x_1^3 - 24 x_1^3 x_2\\
0 & 0\\
x_1^4 - x_2 - 6 x_1^4 x_2 + 6 x_2^2 & x_1^4 - x_2 - 6 x_1^4 x_2 + 6 x_2^2\\
1 & 1
\end{bmatrix}, ~ \lambda(x)= 6x_2^2-1
\]
and
\[
w(x) = \begin{bmatrix}
\frac{1}{6}(5 - 4 x_1^3 + 24 x_1^3 x_2)\\
\frac{1}{6} (1 + 4 x_1^3 - 24 x_1^3 x_2)
\end{bmatrix}.
\]
\end{exa}
\end{comment}

\begin{comment}
\begin{exa}

Consider the OWP with the preference function
\[
f_0(x)= 4x_1-4x_1x_3+x_2+x_4x_5
\]
and
objective functions:
\[
f_1(x)=h(x)+18x_1+11x_2-5x_3-13x_4-3x_5,
\]
\[
f_2(x)=h(x)+19x_2-20x_3-11x_4-18x_5,
\]
\[
f_3(x)=h(x)+11x_1-5x_2-16x_3+20x_4+5x_5,
\]
for $h(x)$ being the Horn's form:
\[
h(x)=(x_1^2+\dots+x_5^2)^2-4(x_1^2x_2^2+x_2^2x_3^2+\dots + x_5^2x_1^2),
\]
over the feasible set
\[
K=\{x\in\r^5:\,1-x_1^2-x_2^2-x_3-x_4-x_5\ge 0\}.
\]
By Algorithm \ref{alg3}, we get $f_{min}=-8.3311$ and obtain that
\[
x = ( -2.2271,
    1.4052,
    0.1749,
    0.4121,
    0.3158),
\]
which gives \[
\lambda =  6.5811,\,w = (1.0000,0.0000,0.0000)
\]
at the relaxation order 5. The computation takes around 178.09 seconds.
\end{exa}

\end{comment}

\section{Applications in multi-task learning}
\label{section6}
Multi-task learning (MTL) is a machine learning approach that trains a model on multiple related tasks simultaneously, utilizing shared information across tasks to enhance performance on each one \cite{multitask survey}. We denote the set of inputs and labels as
$$
X = \{X^{(1)},\dots,X^{(p)}\},~Y = \{y^{(1)},\dots,y^{(q)}\},$$
respectively.
\begin{comment}
The general structure of MTL is illustrated in Figure \ref{mtl}.

\begin{figure}[h!]
    \centering
    \begin{tikzpicture}[node distance=1.5cm and 1.5cm, >=Stealth]
        \node[draw, rectangle, minimum width=1cm, minimum height=1cm] (x1) {$X^1$};
        \node[draw, rectangle, minimum width=1cm, minimum height=1cm, below=0.5cm of x1] (x2) {${X}^2$};
        \node[below=0.3cm of x2] (dots1) {$\vdots$};
        \node[draw, rectangle, minimum width=1cm, minimum height=1cm, below=0.3cm of dots1] (xs) {${X}^p$};
        \node[draw, rectangle, minimum width=1cm, minimum height=2cm, right =1cm of x2, align=center] (processor1) {hidden\\ layer};
        \draw[->] (x1) -- (processor1);
        \draw[->] (x2) -- (processor1);
        \draw[->] (xs) -- (processor1);

        \node[draw, rectangle, minimum width=1cm, minimum height=1cm, right = 3.3cm of x1] (y1) {$y^1$};
        \node[draw, rectangle, minimum width=1cm, minimum height=1cm, right = 3.3cm of x2] (y2) {${y}^2$};
        \node[below=0.3cm of y2] (dots1) {$\vdots$};
        \node[draw, rectangle, minimum width=1cm, minimum height=1cm, right = 3.3cm of xs] (ys) {${y}^m$};
        \draw[->] (processor1) -- (y1);
        \draw[->] (processor1) -- (y2);
        \draw[->] (processor1) -- (ys);

\end{tikzpicture}
    \caption{Illustration of general MTL.}
    \label{mtl}
\end{figure}
\end{comment}
Let $U=\{U_1,\dots,U_n\}$ denote the set of trainable parameters. For $i=1,\dots,m$, where $m$ is the number of tasks, let $g_i(X^{(i)},U)$ and $f_i(g_i(X^{(i)},U),Y)$ denote the output and loss function for the $i$th task, respectively. Finding the parameter set $U$ can be done by solving the following MOP:

\begin{equation}
\label{mtlmin}
    \left\{\begin{array}{cl}
        \min & F(U):= (f_1(U),\dots,f_m(U)) \\
        \st & U\in K,
    \end{array}
    \right.
\end{equation}
where $K$ is the constraining set. In general, $K$ is the Euclidean space, the unit ball or the positive orthant. Generally, we want to solve for the OWP
\begin{equation}
    \label{owpmtl}
    \left\{
    \begin{array}{cl}
        \min& f_0(U) \\
        \st & U\in \mathcal{WP},
    \end{array}\right.
\end{equation}
where the preference function $f_0$ is often the additional criterion or desiderata and the set $\mathcal{WP}$ is the weakly Pareto set of the MOP \reff{mtlmin}.
Based on the number of inputs and number of outputs, the MTL is categorized into the following cases \cite{briefreview}:  the multi-input single-output (MISO), the single-input multi-output (SIMO), and the multi-input multi-output (MIMO). In the following, we use our algorithms to solve the MISO, MIMO.

\begin{exa}
We consider the MISO case, where multiple data sources are mapped to the same label $y$, and each task involves predicting the common label
$y$ based on a single data source.
We utilize the loss function outlined in \cite{briefreview}:
\[ f_i(X^{(i)},u,y)=\|X^{(i)}u-y\|^2,~i\in[p],
\]
where $X=\{X^{(1)},\dots,X^{(p)}\}$ denotes the set of data, $X^{(i)}\in\r^{n_1\times n_2}$ denotes the $i$th data source, $y\in \r^{n_1}$ is the common label and
$u \subseteq \r^{n_2}$ denotes the trainable parameters. We further assume that the parameters are in the nonnegative orthant.
%$$
%K=\{U\in\r^{n_2\times p}:\, \sum\limits_{i=1}^p\sum\limits_{j=1}^{n_2} u_i_j\le 1\}.
%$$
Consider the case that
%$p=n_1=n_2=2$ and
$p=5,\,n_1=n_2=10$, the preference function is
$$
f_0(u) = \|u\|_2^2,
$$
\begin{comment}
the input data is
\[
X^{(1)}=
\begin{bmatrix}
 5   &  6&     8  &  -1   &  6  &   5    & 1   & 10  &   6   &  4\\
    10   &  6    & 4   &  9  &   5    & 5   &  7   &  3  &   1  &   5\\
     9   &  3    & 0  &   4 &    2  &   3  &   9  &   3  &   3   &  3
\end{bmatrix}^T,
\]
\[
X^{(2)} = \begin{bmatrix}  0  &   2  &   0   &  4   &  2  &   2   &  4  &   0    & 4   &  7\\
     1   &  8  &  -1  &   3 &   -1 &    1   & -1   &  1    & 0   &  2\\
     1  &   0    & 6   &  9  &  10 &    1 &    4   &  3   &  5   &  2\end{bmatrix}^T,\]
     and the label is
\[
y = \begin{bmatrix}0   &  2  &   1  &   0   &  5  &   2 &    5   &  4   &  0   &  4\end{bmatrix}^T.
\]
\end{comment}
and $X^{(1)},\dots,X^{(5)}\in \r^{10\times 10}$ and $y\in \r^{10}$ are randomly generated, with all their entries following the standard normal distribution.
%The objective functions are $F=(f_1,f_2,f_3,f_4,f_5)$ for
%\[
%f_i(u)=\|X^{(i)}u-y\|_2^2 ~\text{ for }~ i=1,\dots, 5,\]
%and
The feasible set is
\[
K = \{u\in\r^{10}:\,u\ge 0\}.
\]
%\cred{where is the target???}
Then, we have that $ C^{\prime}_1(u)=I_{10}$ and
\[
~ \lambda(u,w)= \sum\limits_{i=1}^5w_i\nabla f_i(u).
\]
By Algorithm \ref{alg1}, we get $f_{min}= 0.6451$ and obtain that
\[
w = ( 0.3771  ,  0.1320   , 0.1597  ,  0.2957, 0.0355),
\]
\[
u = (  0.3744,
    0.3390,
    0.0785,
    0.0000,
    0.0433,
    0.3766,
    0.0656,
    0.0001,
    0.3517,
    0.3349)
\]
%and obtain that
%\[
%u =(  0.2459,
%    0.0258,
%    0.2036,0,0,0), ~w = (1,0)
%\]
at the relaxation order $k=2$. The computation takes around 13.03 seconds.

\end{exa}

\begin{exa}

We consider the MIMO case,  involving multiple data sources and targets within an autoencoder framework designed to compress and reconstruct the input data. It is important to note that, in an autoencoder process, the input sources and targets are identical. Each task aims to reconstruct the input data. The process is constructed as follows:
\begin{center}
\begin{tikzpicture}[scale=1]
 % Input layer (X)
    \draw[fill=gray!30] (-4,1) rectangle (-2,-1);
    \node at (-3, 0) {$X$};
    \node at (-3, -1.2) {Input Layer};

    % Encoder label
    \node at (-2.7, 1.8) {Encoder};

    % Code (h)
    \draw[fill=gray!10] (-1, 0.5) rectangle (1,-0.5);
    \node at (0, 0) {$ $};
    \node at (0, -1) {Code};

    % Output layer (X')
    \draw[fill=gray!50] (2,1) rectangle (4,-1);
    \node at (3, 0) {$\hat{X}$};
    \node at (3, -1.2) {Output Layer};

    % Decoder label
    \node at (2.7, 1.8) {Decoder};

    % Arrows
    \draw[->, thick] (-2,0) -- (-1,0); % Arrow from X to h
    \draw[->, thick] (1,0) -- (2,0);   % Arrow from h to X'

    % Dashed boxes around encoder and decoder
    \draw[dashed] (-4.5,1.5) rectangle (-0.5,-1.5);
    \draw[dashed] (0.5,1.5) rectangle (4.5,-1.5);
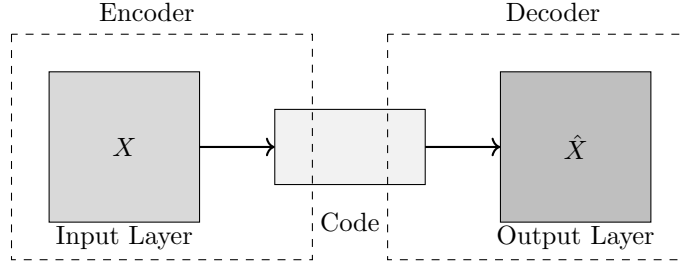
\node at (0,-2){Figure: Illustration of Autoencoder Process};
\end{tikzpicture}
\end{center}

For inputs $X = \{X^{(1)},\dots,X^{(p)}\}$, where $X^{(i)}\in\r^n$ for $i\in[p]$, we denote the $j$th entry of $X^{(i)}$ as $X^{(i)}_j$. The encoder process consists of one ReLU layer defined by the elementwise operation
$$
\operatorname{ReLU}(x)=\max (0,x),
$$
%\cred{which is equivalent to the constrained  polynomial optimization problem
%$$
%\left\{\begin{array}{ll}\max & t\\
%\st & t(t-x)=0,\,t\ge 0,\,t\ge x,
%\end{array}\right.
%$$
and one linear layer defined by
$$
\operatorname{Linear}(x,A,b)=Ax+b,
$$
%We note that for each $x$, ReLU$(x)$ is the optimal value of the polynomial optimization problem
%\[\left\{
%\begin{array}{cl}
%     \max & t \\
%     \st& t(t-x)=0,\,t\ge 0,\,t\ge x,
%\end{array}
%\right.\]
where $A\in\r^{n\times n}$ and $b\in\r^n$ satisfying $\|A\|_F^2,\|b\|^2\le 1$ are to be determined. The decoder process consists of one
Sigmoid layer defined by the elementwise operation \cite{chebyapproxi}
\[
\sigma(x)=\operatorname{tanh}(x)\approx x  -
\frac{x^3}{3}
+\frac{2x^5}{15}.
\]
Here, the output is the composition of the above operations, i.e.,
$$
g_i(X^{(i)},U)=\sigma(X^i)\circ \operatorname{Linear}(X^{(i)},U)\circ \operatorname{ReLU}(X^{(i)}),
$$
and  the  loss function is
$$
f_i(g_i(X^{(i)},U),X) = \|g_i(X^{(i)},U)-X^{(i)}\|^2,
$$
 where  $U=\{A,b\}$ is the trainable parameter set.
Let $\hat{X^i}=g_i(X^{(i)},U)$ represent the output of the autoencoder process, and define the preference function $f_0$ as
$$
f_0(U):=\sum^p_{i=1}\sum^n_{j=1} X_{j}^{(i)}-\hat{X}_j^{(i)}.
$$
The
feasible set is
\[
K = \{\,1-\|A\|_F^2\ge 0,\,1-\|b\|^2\ge 0\}.
\]
Consider the case that $n=4$ and $p=10$,
%let $u$ denote the vectorized trainable parameters, which is
%\[
%u=\begin{bmatrix}A_{11}&A_{12}&A_{21}&A_{22}&b_1&b_2\end{bmatrix}^T%.
%\]
we have that
\[
K=\{(A,b)\in\r^{4\times 4}\times \r^4:\, 1-\|A\|_F^2\ge 0,\,1-\|b\|^2\ge 0\}.
\]
%\cred{We note that here $A^2_{ij}$ and $b^2_i$ denote the square of $A_{ij}$ and $b_i$ for $i,j \in \{1,2\}.$}
%$$
%C(A,b) = \begin{bmatrix}
%    -2A_{11} & -2A_{12} & -2A_{21} &-2A_{22} & 0&0&1-A_{11}^2-A_{12}^2-A_{21}^2-A_{22}^2\\
%    0&0&0&0&-2b_1&-2b_2&1-b_1^2-b_2^2
%\end{bmatrix}^T,
%$$
%and $L(A,b)C(A,b)=I_2$ for
%\[
%L(A,b)=\begin{bmatrix}
%    -\frac{1}{2}A_{11} & -\frac{1}{2}A_{12} & -\frac{1}{2}A_{21} & -\frac{1}{2}A_{22}&0&0&1\\
%    0&0&0&0&-\frac{1}{2}b_1&-\frac{1}{2}b_2&1
%\end{bmatrix}.
%\]
The matrix $C^{\prime}_1(U)$ is computed as
%\[
%C(u) = \begin{bmatrix}
%    -2A_{11} & -2A_{12} & -2A_{21} &-2A_{22} & 0&0&1-A_{11}^2-A_{12}^2-A_{21}^2-A_{22}^2 & 0\\
%   0&0&0&0&-2b_1&-2b_2&0&1-b_1^2-b_2^2
%\end{bmatrix}^T,
%\]
\[
C^{\prime}_1(U)=\begin{bmatrix}
    -\frac{1}{2}\operatorname{vec}(A)&0\\
    0&-\frac{1}{2}e^Tb
\end{bmatrix},
\]
where $\operatorname{vec}(A)$ represents the vectorized form of matrix $A$.
Using Algorithm \ref{alg1}, we solve for the parameters
$
A$ and
$b$ with 100 sets of randomly generated $X^1,\dots,X^{10}$, where $X^i_j\in [-1,1]$ and  obey the standard normal distribution. The average minimum value of
$f_0$
  is found to be 0.1280, with an average computation time of 107.02 seconds.
\end{exa}

\section{Conclusions and discussions}\label{sec:dis}

In this paper, we propose three algorithms to solve polynomial optimization problem
over the weakly Pareto set of the convex MOP, based on different reformulations of the weakly Pareto set. Numerical experiments
are conducted to show the efficiency of these methods. We also apply our algorithms to solve the multi-input single-output problem and the multi-input multi-output problem. There are many interesting questions for future research. For instance, when the MOP \reff{mop} is nonconvex, it is typically hard to give an efficient characterization for the weakly Pareto set. How to efficiently approximate this set is an important question; see \cite{vmag}.

\section*{Acknowledgements}
This work is partially supported by the NSF grant DMS-2110780.

\end{document}